\documentclass[11pt]{article}
\usepackage[letterpaper, portrait, margin=1in]{geometry}

\usepackage{graphicx}
\usepackage{xcolor}
\usepackage{float}

\usepackage[round]{natbib}

\usepackage{amsmath}
\usepackage{amssymb}
\usepackage{amsfonts}
\usepackage{amsthm}

\usepackage{csquotes}
 \usepackage[ruled,vlined]{algorithm2e}
 
 \usepackage{enumitem}
\setlist{nolistsep}

\allowdisplaybreaks[1]

\newtheorem{theorem}{Theorem}[section]

\newtheorem{lemma}[theorem]{Lemma}
\newtheorem{remark}{Remark}[section]

\newtheorem{manualtheoreminner}{Theorem}
\newenvironment{manualtheorem}[1]{%
  \renewcommand\themanualtheoreminner{#1}%
  \manualtheoreminner
}{\endmanualtheoreminner}

\newtheorem{manuallemmainner}{Lemma}
\newenvironment{manuallemma}[1]{%
  \renewcommand\themanuallemmainner{#1}%
  \manuallemmainner
}{\endmanuallemmainner}

\newcommand{\ignore}[1]{}

\def\defi{:=}
\def\diff{\mathop{}\!\mathrm{d}}

\title{{\bf Revenue Management Under the Markov Chain Choice Model with Joint Price and Assortment Decisions}}

\author{
{\bf Anton J. Kleywegt},  {\bf Hongzhang Shao} \\ % \thanks{}
School of Industrial and Systems Engineering \\
Georgia Institute of Technology \\
Atlanta, Georgia 30332-0205 \\
}

\date{
}

\begin{document}

\maketitle

\begin{abstract}
Finding the optimal product prices and product assortment are two fundamental problems in revenue management. 
Usually, a seller needs to jointly determine the prices and assortment while managing a network of resources with limited capacity. 
However, there is not yet a tractable method to efficiently solve such a problem. 
Existing papers studying static joint optimization of price and assortment cannot incorporate resource constraints. 
Then we study the revenue management problem with resource constraints and price bounds, where the prices and the product assortments need to be jointly determined over time. 
We showed that under the Markov chain (MC) choice model (which subsumes the multinomial logit (MNL) model), we could reformulate the choice-based joint optimization problem as a tractable convex conic optimization problem. 
We also proved that an optimal solution with a constant price vector exists even with constraints on resources. 
In addition, a solution with both constant assortment and price vector can be optimal when there is no resource constraint. 
\end{abstract}

\newpage

%%%%%%%%%%%%%%%%%%%%%%%%%%%%%%
%%%%%%%%%%%%%%%%%%%%%%%%%%%%%%

%--------------------------------------------------
%	Section 
%--------------------------------------------------

\section{Introduction}
\label{sec:introduction}

Finding the optimal product prices and product assortment are the two most fundamental problems in revenue management. 
In many applications, the seller needs to determine the product prices and product assortment \emph{jointly}. 
Usually, each product uses a specific combination of resources. 
Thus, the seller needs to optimally choose the prices and assortment to maximize his profit, market share, or revenue while considering the capacity limited on resources.
For example, consider an airline designing itineraries over a network of flight legs. 
Different itineraries correspond to different products and seats on different flight legs. 
The airline needs to first choose the itineraries' prices at the beginning of the planning horizon, then adjust their availability over time. 
In other scenarios such as e-commerce, both the prices and availability of products can be adjusted over time. 

Although such joint decision problems with resource constraints are pretty standard in practice, there is not yet an efficient method to solve them. 
At present, we have methods to efficiently solve: 
\begin{enumerate}[label=(\roman*)]
\item network revenue management problems where prices of products are exogeneous and cannot be controlled (see \cite{gallego2011general, feldman2017revenue}, etc.); 
\item pricing problems with resource constraints where products are assumed to be always available over time (see \cite{song2007demand, keller2014efficient, dong2019pricing}, etc.); 
\item static joint optimization problem of product prices and product assortment, but without resource constraints (see \cite{wang2012joint, wang2017joint}, etc.). 
\end{enumerate}
However, none of the existing efficient solution methods can be extended to solve the joint optimization problem of prices and assortment while incorporating resource constraints. 

This paper presents a tractable approach to solving the revenue maximization problem under the Markov chain (MC) choice model with resource constraints, where the product prices and product assortment need to be chosen simultaneously. 
We formulate both the static optimization problem (\ref{eqn:sp1}) and the choice-based optimization problem (\ref{eqn:cbp1}) and consider both resource constraints and upper and lower bounds of prices. 
We show that both problems can be solved with a convex conic reformulation. 
We then provide the necessary and sufficient conditions for an optimal solution to exist for each problem. 
For (\ref{eqn:cbp1}), we show that an optimal solution with a constant price vector exists even when there are constraints on resources. 
Thus, there is no need to re-adjust prices during the planning period. 
In addition, when there is no resource constraint, a solution with a constant assortment and price vector can be optimal. 

\vspace*{20pt}

\begin{table}[h]
\centering
{\scriptsize
\begin{tabular}{ l || l | l }
\hline
Problem Type:  & Without Resource Constraints & With Resource Constraints \\ 
 & (Static Optimization) & (Revenue Management with Limited Resources) \\
 & & \\
 \hline
 \hline
Joint Assortment and & \textbf{This Paper} (MC/MNL) & \textbf{This Paper} (MC/MNL) \\ 
Price Optimization & \cite{wang2017joint} (NL/MNL) & \\
 & \cite{wang2012joint} (MNL) & \\
 & * \cite{gallego2014constrained} (NL/MNL) & \\
 & * \cite{davis2013assortment} (MNL) & \\
 & & \\
 \hline
 \hline
Assortment Optimization & \textbf{This Paper} (MC/MNL) & \textbf{This Paper} (MC/MNL) \\ 
 & \cite{feldman2017revenue} (MC/MNL) & \cite{feldman2017revenue} (MC/MNL) \\
 & \cite{blanchet2016markov} (MC/MNL) & \cite{gallego2011general} (MNL) \\
 & & ... \\
 & \cite{wang2017joint} (NL/MNL) & \\
 & \cite{wang2012joint} (MNL) & \\
 & * \cite{gallego2014constrained} (NL/MNL) & \\
 & * \cite{davis2013assortment} (MNL) & \\
 & ... & \\
 \hline
Price Optimization  & \textbf{This Paper} (MC/MNL) & \textbf{This Paper} (MC/MNL) \\ 
 & \cite{dong2019pricing} (MC/MNL) & \cite{dong2019pricing} (MC/MNL) \\
 & \cite{keller2014efficient} (GAM/MNL) & \cite{keller2014efficient} (GAM/MNL) \\
 & \cite{li2011pricing} (NL/MNL) & \cite{li2011pricing} (NL/MNL) \\
 & & ... \\
 & \cite{wang2017joint} (NL/MNL) & \\
 & \cite{wang2012joint} (MNL) & \\
 & * \cite{gallego2014constrained} (NL/MNL) & \\
 & * \cite{davis2013assortment} (MNL) & \\
 & ... & \\ \hline
\end{tabular}
}
\vspace*{10pt}
\caption{Type of optimization problems covered in several papers. The * simbol indecates that a discretization of prices is needed. }
\label{tab:paper-compare}
\end{table}

Closest to our paper are \cite{feldman2017revenue} and \cite{dong2019pricing}. 
\cite{feldman2017revenue} studies both the static assortment planning problem and the choice-based network revenue management problem under the MC model. 
They provided both the reduced problem and the algorithm for recovering the optimal solution to the choice-based approximation problem. 
\cite{dong2019pricing} studied the pricing problem under the MC model and showed that the problem could be solved efficiently either through dynamic programming or by solving a convex program. 
Our paper is a generalization based on the two papers. 
The revenue management problem in \cite{feldman2017revenue} and pricing problem in \cite{dong2019pricing} are both special cases of problem (\ref{eqn:cbp1}). 
The MC model was first systematically examined by \cite{blanchet2016markov}. 
The MC model in our paper is similar to models considered in \cite{feldman2017revenue} and \cite{dong2019pricing}, adjusted to allow both assortment decisions and price changes. 
In addition, the conic reformulation method in this paper extends the reformulation method we developed in \cite{shao2020tractable}, which is established based on papers studying variable change methods to solve pricing problems. 

This paper is organized as follows: 
In Section \ref{sec:literature}, we review the related literature. 
In Section \ref{sec:model}, we introduce the Markov chain choice model used in this paper. 
In Section \ref{sec:joint}, we formulate the static joint optimization problem of product prices and assortment (\ref{eqn:sp1}). 
We show that an optimal solution to (\ref{eqn:sp1}) can be found by solving a convex conic program, and such an optimal solution always exists. 
In Section \ref{sec:network}, we extend the tractability result on (\ref{eqn:sp1}) to a choice-based network revenue management problem (\ref{eqn:cbp1}) with joint price and assortment decisions. 
We also show that allowing prices to change over time does not improve revenue. 
Thus, it is enough to use a constant price vector over time. 
We conclude in Section \ref{sec:conclusion}.

%--------------------------------------------------
%	Section 
%--------------------------------------------------

\section{Related Literature}
\label{sec:literature}

\emph{\underline{Assortment Planning \& Network Revenue Management}:} Consider a seller who manages a set of products. 
Different products have different profit margins and demands.
Usually, products are substitutable, and whether one product is offered affects the demands of other products. 
Thus, to maximize the total revenue, the seller must decide which subset of products to offer during a selling horizon. 
Such an optimization problem is usually called an assortment planning problem.
It is one of the most fundamental problems in modern choice-based revenue management literature. 

An essential extension to the assortment planning problem is the network revenue management problem, where the seller manages a network of resources, each of which has limited capacity. 
Each product uses a specific combination of resources. 
When the seller sells a product, he generates revenue and consumes the capacities of the resources used by the product. 
One common approach to solving the network revenue management problem is to formulate deterministic approximations, assuming that customer choices take on their expected values. 
Such an approximation is usually called a ``choice-based optimization problem''. 
It first appeared in \cite{gallego2004managing}, and was discussed by following papers such as \cite{liu2008choice} and \cite{bront2009column}. 
One obstacle in solving a choice-based optimization problem is that its number of decision variables grows exponentially as the number of products increases. 
To overcome this obstacle, \cite{gallego2011general} showed that the choice-based optimization problem could equivalently be reformulated as a much smaller ``sales-based optimization problem'' under the multinomial logit (MNL) choice model. 
\cite{feldman2017revenue} then extended this result to the choice-based optimization problem under the Markov chain (MC) choice model (which subsumes the MNL model). 
We refer the readers to \cite{strauss2018review} for an in-depth review of literature on choice-based revenue management. 

\emph{\underline{Price Optimization}:} Another fundamental problem in choice-based revenue management is the pricing problem. 
Usually, the price of a product affects the profit of the seller in two ways: 
(1)~The price affects the demand for the product (and the demand for other products).
(2)~The price affects the revenue and cost per unit product sold, and thus the profit margin. 
Both effects can be captured by a choice model with parameters that may depend on the product. 
However, the pricing problem is hard to solve under most choice models in its original form. 
For instance, \cite{hanson1996optimizing} studied the problem of maximizing the expected revenue under the multinomial logit (MNL) model. 
They observed that the expected profit is not concave in price. 
Since then, researchers have been actively developing efficient solution methods for pricing problems. 
Many tractable approaches have been introduced in the last two decades. 

A critical approach to solving the pricing problems efficiently is to perform a change in variables. 
Early papers along this line are \cite{song2007demand} and \cite{dong2009dynamic}, who considered pricing and inventory control problems under the multinomial logit (MNL) model and showed that the expected revenue is concave in the market shares. 
In order words, the pricing problem under the MNL model can be solved efficiently by using the market shares of the products as decision variables. 
(Both works assumed that all products' price sensitivity parameters are the same.) 
\cite{li2011pricing} extended the concavity results to the nested logit (NL) model (which subsumes the MNL model) and showed that the expected profit is concave in the market shares if (1) the price-sensitivity parameters are identical for all the products within a nest, and (2) the nest coefficients are restricted to be in the unit interval. 
As a special case, the expected profit under a general MNL model (with asymmetric price-sensitivity parameters) is concave in market shares. 
\cite{keller2014efficient} discussed similar concavity results under the general attraction demand model (GAM, which also subsumes the MNL model) and showed that constraints such as price bounds and price ladders could be added to the pricing problem as linear constraints in the market shares. 
\cite{zhang2018multiproduct} discussed the multi-product pricing problem under the generalized extreme value (GEV) models (which subsumes the NL model). 
They showed that the problem could be formulated as a convex program with homogeneous price sensitivity parameters. 
Recently, \cite{dong2019pricing} formulate the pricing problem under the Markov chain (MC) choice model. 
They showed that the pricing problem could be solved as a dynamic program. 
They also presented a market-share-based approach to solving the dynamic pricing problem with a single resource.

Besides the variable change approach, there are two other approaches. One approach to solving the pricing problems under the MNL and NL model is to study the first-order optimality condition. 
\cite{anderson1992multiproduct}, \cite{besanko1998logit} and \cite{aydin2000product} observed that under an MNL model with homogeneous price sensitivity parameters, the profit margin (price minus cost) is constant across all the products at the optimality of the expected profit. 
\cite{aydin2008joint} and \cite{akcay2010joint} then pointed out that under such scenarios, the profit function is uni-modal concerning the markup. Its unique optimal solution can thus be found by solving the first-order conditions. 
\cite{gallego2014multiproduct} extended these results and showed that under the general NL model, the adjusted markup (price minus cost minus the reciprocal of price sensitivity) is constant for all the products within a nest at optimality. 
They then defined an adjusted nest-level markup and showed it invariant for all nests, reducing the pricing problem to a single-dimensional optimization problem. 
This one-dimensional problem has a local maximum when (1) the difference between price sensitivity parameters is bounded by a particular value, or (2) the dissimilarity parameter is greater than 1. 
These conditions are more general than \cite{li2011pricing}. 
(\cite{wang2012joint} also showed the ``constant adjusted markup'' property for pricing problems under the MNL model.) 
\cite{li2015d} further extended this result to the pricing problem under the multi-level nested logit model. 
They showed that the problem could be reduced to single-dimensional, and the optimal solution is unique under generalized conditions. 
\cite{zhang2018multiproduct} discussed the multi-product pricing problem under the generalized extreme value models with homogeneous price sensitivity parameters and provided similar results. 

The other approach limits price candidates to a small, finite set. 
For instance, \cite{davis2013assortment} studied assortment planning problems under the MNL model with totally uni-modular constraints and showed that such problems could be solved efficiently as linear programs. 
One special case of their optimization problem is the pricing problem with a finite price menu. Specifically, they make several copies of a product, each with a different price. They then use a cardinality constraint (a totally uni-modular constraint) to ensure that exactly one out of these copies will be chosen. 
\cite{davis2013assortment} also showed that in such a pricing problem, the quality-consistent constraints (price ladders) could be imposed as totally uni-modular constraints. 
\cite{gallego2014constrained} then showed that assortment planning problems under the NL model with cardinality constraints could be solved efficiently by solving a linear program. Thus, it is also possible to efficiently solve the pricing problem under the NL model with finite price candidates. 
However, it was not shown that their method could be extended to quality-consistent constraints or totally uni-modular constraints in general. 
\cite{davis2017pricing} then followed a different, fixed-point-based approach to solve the pricing problem under the NL model with quality consistency constraints and showed that the problem could be solved by solving a linear program. 

\emph{\underline{Joint Optimization of Prices and Assortment}:} 
While our method is related to the variable-change approach for pricing problems, existing approaches to solve the joint optimization of price and assortment are established based on the other two methods mentioned above. 
For example, \cite{wang2012joint} studied the joint optimization of assortment selection and pricing under the multinomial logit (MNL) choice model and showed that the problem could be efficiently solved by finding the fixed point of a single-dimensional function. 
\cite{wang2017joint} then extended the tractability result to the joint optimization problem of assortment and price under the tree logit model. 
Both papers rely on properties that were established upon the first-order conditions. Thus, while they can work with cardinality constraints, their methods cannot incorporate other typical constraints, such as resource constraints and price bounds. 

Another existing approach to solving the joint optimization problems is to limit price candidates to a small finite set and reduce the problems to classical assortment optimization problems. 
For instance, \cite{davis2013assortment} studied assortment planning problems under the MNL model with totally uni-modular constraints and showed that such problems could be solved efficiently as linear programs. 
One special case of their optimization problem is the pricing problem with a finite price menu. Specifically, they make several copies of a product, each with a different price. They then use a cardinality constraint (which is a totally uni-modular constraint) to ensure that exactly one out of these copies will be chosen. 
\cite{gallego2014constrained} then showed that assortment planning problems under the NL model with cardinality constraints could be solved efficiently by solving a linear program. 
Thus, it is also possible to efficiently solve the joint optimization problem under the NL model. 
Again, it was not shown that any of the methods above could be extended to incorporate resource constraints. 

Besides, there are several downsides to discretizing prices: 
(1) Recall that our approach guarantees that there is an optimal solution with a constant price over time. 
Such a property is very practical as it is usually much harder to frequently change the prices than to change the availabilities of products. 
After discretizing prices, however, there is no guarantee that the solution will have constant prices for products (see Remark \ref{rmk:const-price}). 
(2) Discretizing prices makes the problem larger and thus harder to solve. 
For example, if we create ten price candidates (which is a small number) for each product, then the scale of the problem becomes (at least) ten times larger. 
A ride-sharing service operator may need to update prices for millions of location pairs every a few minutes and may need the pricing problem to be solved within seconds. 
Making the problem ten times larger can be fatal under this scenario. 
(3) It is also questionable how these price candidates should be selected in advance, especially without knowing where better prices may fall. 
Thus, our method also provides a good alternative approach to solving the static joint optimization problem that overcomes the downsides above. 

\begin{remark}
\label{rmk:const-price}
As an example that prices may not stay constant over time after discretizing prices, consider an assortment of one product with two price candidates. 
The first price is relatively lower but leads to higher demand. 
The second price is relatively higher but leads to low demand. 
Offering either price is very profitable. 
However, we do not have enough resources if we offer the first price throughout the planning period. 
Thus, the optimal solution is to offer each price for a fraction of the time. 
(Note that, by the concavity of the pricing problem (see Section \ref{sec:network}), the single optimal price should be in between the two prices above.)
\end{remark}

%------------------------------------------------------------------------
%	Section
%------------------------------------------------------------------------

\section{Markov Chain Choice Model}
\label{sec:model}

Consider a set of $J$ products, denoted by $\mathcal{J}$. 
Each product $j \in \mathcal{J}$ has a price $x_{j}$, the value of which can be chosen by the seller.
The price of a product can affect the profit of the seller in two ways:
(1)~The price affects the demand for the product (and the demand for other products). 
(2)~The price affects the seller's revenue, and thus profit margin, per unit product sold. 
Let $x \defi (x_{j} : \ j \in \mathcal{J})$. 
In this paper, we assume that $x_{j}$ has to be chosen within a closed interval $[\underline{x}_{j}, \overline{x}_{j}]$, as demand models are usually calibrated with a limited range of attribute-value data. It would be unwise to select attribute values much outside this range. 

The seller also determines which products it offers to customers. Let $\mathcal{A} \subseteq \mathcal{J}$ denote the assortment the seller chooses to offer. 
Given an offered assortment $\mathcal{A}$ and a price vector $x$, the probability that a product $j$ is chosen by a customer, denoted by $P_{j}^{\mathcal{A}}(x)$, is specified by a choice model with parameters that may depend on the product.
In this paper, we let the choice probabilities $P_{j}^{\mathcal{A}}(x)$ be given by a Markov chain choice model. 
The model here is similar to models considered in \cite{feldman2017revenue} and \cite{dong2019pricing}, adjusted to make provision for both the assortment decision and the price attributes.
The choice probabilities are given by
\begin{align*}
P_{j}^{\mathcal{A}}(x) 
		& \ \ = \ \ \mathbf{1}(j \in \mathcal{A}) Q_{j}(x_{j}) V_{j}^{\mathcal{A}}(x) 
		& j \in \mathcal{J}
\end{align*}
where $V^{\mathcal{A}}(x) = (V_{j}^{\mathcal{A}}(x) \, : \, j \in \mathcal{J})$ is the \emph{unique} solution to the system of equations
\begin{align*}
V_{j}^{\mathcal{A}}(x) 
		& \ \ = \ \ \theta_{j} + \sum_{i \in \mathcal{A}} 
		\left[1 - \mathbf{1}(i \in \mathcal{A}) Q_{i}(x_{i})\right] 
		\rho_{ij} V_{i}^{\mathcal{A}}(x) 
		& j \in \mathcal{J}
\tag{$\mathsf{Balance}$}
\label{eqn:balance}
\end{align*}
in which the parameters $\theta \defi (\theta_{j} : \ j \in \mathcal{J})$, $\rho \defi (\rho_{ij} : \ i,j \in \mathcal{J})$ and functions $(Q_{j}(x_{j}) : \ j \in \mathcal{J})$ satisfy the following assumptions:
\begin{enumerate}
\item
$\rho \geq 0$, $\theta \geq 0$ and $\sum_{j \in \mathcal{A}} \theta_{j} = 1$; 
% and $\sum_{j \in \mathcal{A}} \theta_{j} > 0$ (otherwise the problem is trivial); 
\item
$I - \rho$ is nonsingular ($I$ is an identity matrix with proper dimensionality); 
\item
$Q_{j}(x_{j}) \in [0,1]$ for all any $x_{j} \in [\underline{x}_{j}, \overline{x}_{j}]$.
\end{enumerate}

\begin{remark}
\label{rmk:balance-unique}
Note that the three assumptions here are sufficient to guarantee that the (\ref{eqn:balance}) equations have a unique solution. 
First, $(I - \rho)^{-1} = I + \sum_{n=1}^{\infty} (\rho)^{n}$, which means $I - \rho$ is non-singular if and only if $\sum_{n=1}^{\infty} (\rho)^{n} < \infty$.
Second, let 
\begin{align*}
(\rho')_{ij}^{\mathcal{A}}(x) 
		& \ \ \defi \ \ \left[1 - \mathbf{1}(i \in \mathcal{A}) 
		Q_{i}(x_{i})\right] \rho_{ij}
		& i,j \in \mathcal{J} \\
(\rho')^{\mathcal{A}}(x) 
		& \ \ \defi \ \ ((\rho')_{ij}^{\mathcal{A}}(x) \, : \, i,j \in \mathcal{J})
\end{align*}
Since $Q_{j}(x_{j}) \in [0,1]$ for every $j \in \mathcal{J}$, we have 
\begin{align*}
& \sum_{n=1}^{\infty} ((\rho')^{\mathcal{A}}(x))^{n} 
		\ \ \leq \ \ \sum_{n=1}^{\infty} (\rho)^{n} 
		\ \ < \ \ \infty \\
\implies \quad
& \sum_{n=1}^{\infty} ((\rho')^{\mathcal{A}}(x)^{\top})^{n} 
		\ \ < \ \ \infty 
\end{align*}
Thus, the inverse of $I - (\rho')^{\mathcal{A}}(x)^{\top}$ exists, and
the (\ref{eqn:balance}) equations has a solution
\begin{align*}
V^{\mathcal{A}}(x) 
		\ \ = \ \ \left(I - (\rho')^{\mathcal{A}}(x)^{\top}\right)^{-1} \theta
\end{align*}
that is unique. 
\end{remark}

\begin{remark}
\cite{dong2019pricing} provided the following interpretation to the MC model:
Each arriving customer first visits product $j$ with probability $\theta_{j}$.
A customer visiting product $j$ purchases the product with probability $\mathbf{1}(j \in \mathcal{J}) Q_{j}(x_{j})$, or transition from the product with probability $1 - \mathbf{1}(j \in \mathcal{J}) Q_{j}(x_{j})$.
When a customer choose to transition from product $j$, she either transition to another product $i$ with probability $\rho_{ji}$, or transition to the no purchase option (and leaves the system) with probability $1 - \sum_{i \in \mathcal{A}} \rho_{ji}$.
In this way, each arriving customer transitions between different products until purchasing one of the products or deciding to leave without making a purchase.
Thus, $V_{j}^{\mathcal{A}}(x)$ gives the expected number of times that a customer visits product $j$.

\noindent However, as \cite{dong2019pricing} stated, we should not view the MC model as a faithful model of the thought process of the customers when they purchase a product. The value of the MC model comes from the facts that (1) this model is compatible with the random utility maximization principle (see Theorem 1 in \cite{dong2019pricing}), (2) its parameter values can be set so that the purchase probabilities under it become identical to the generalized attraction model, which subsumes the MNL model (see Lemma 2 in \cite{dong2019pricing}), and (3) it yields tractable optimization problems (see \cite{blanchet2016markov}, \cite{feldman2017revenue} and \cite{dong2019pricing}). 
\end{remark}

In this section, we let
\begin{align*}
Q_{j}(x_{j}) 
		& \ \ \defi \ \ \exp\left( \alpha_{j} - \beta_{j} x_{j} \right)
		& j \in \mathcal{J}
\end{align*}
We assume that $\beta_{j} > 0$ for each $j \in \mathcal{J}$. A higher price (and thus a higher profit margin) decrease customers' willingness to purchase the product and reduces the demand.

%------------------------------------------------------------------------
%	Section
%------------------------------------------------------------------------

\section{Joint Optimization of Price and Assortment}
\label{sec:joint}
In this section, we study the static problem of maximizing the expected profit per customer arrival by controlling product prices and assortment selection. 
For each $j \in \mathcal{J}$, let $\psi_{j}$ denote the unit cost of product $j$. 
The resulting static optimization problem is
\begin{subequations}
\begin{align}
\max_{v, \ x, \ a} \quad
	& \sum_{j \in \mathcal{J}} \left(x_{j} - \psi_{j}\right) 
	a_{j} \exp\left( \alpha_{j} - \beta_{j} x_{j} \right) v_{j}
\tag{$\mathsf{SP}_1$}
\label{eqn:sp1} \\
\text{s.t.} \quad
& v_{j} \ \ = \ \ 
		\theta_{j} + \sum_{i \in \mathcal{J}} 
		\left[1 - a_{i} \exp\left( \alpha_{i} - \beta_{i} x_{i} \right) \right] \rho_{ij} v_{i}
		& \forall \ j \in \mathcal{J} \nonumber \\
& a_{j} \ \ \in \ \ \{0,1\}
		& \forall \ j \in \mathcal{J} \nonumber \\
& x_{j} \ \ \leq \ \ \overline{x}_{j}
		& \forall \ j \in \mathcal{J} \nonumber \\
& x_{j} \ \ \geq \ \ \underline{x}_{j}
		& \forall \ j \in \mathcal{J} \nonumber 
\end{align}
\end{subequations}
where $v \defi (v_{j} : \ j \in \mathcal{J})$ and $a \defi (a_{j} : \ j \in \mathcal{J})$. The first constraint in (\ref{eqn:sp1}) is consistent with the (\ref{eqn:balance}) equations, except the indicator functions $\mathbf{1}(j \in \mathcal{J})$ are replaced by binary decision variables $a_{j}$. Here $a_{j}=1$ indicates that product $j$ is in the chosen assortment, and is available to be purchased. Otherwise, $a_{j}=0$. Thus, $a_{j} \exp\left( \alpha_{j} - \beta_{j} x_{j} \right) v_{j}$ gives the overall probability that product $j$ is purchased by a customer.

Problem (\ref{eqn:sp1}) is not a convex optimization problem.
One of the reasons is that $a_{j}$ are binary variables.
Also, the first equality constraint is nonlinear. 
We will (1) relax $a_{j}$ into continuous variables, show that an optimal solution of the relaxation provides an optimal solution of (\ref{eqn:sp1}), (2) reformulate the relaxation as a convex optimization problem, and (3) then provide necessary and sufficient conditions for this convex optimization problem to have an optimal solution. 

First, we relax variable $a$ to obtain the following optimization problem: 
\begin{subequations}
\begin{align}
\max_{v, \ x, \ a} \quad
	& \sum_{j \in \mathcal{J}} \left(x_{j} - \psi_{j}\right) 
	a_{j} \exp\left( \alpha_{j} - \beta_{j} x_{j} \right) v_{j}
\tag{$\mathsf{SP}_2$}
\label{eqn:sp2} \\
\text{s.t.} \quad
& v_{j} \ \ = \ \ 
		\theta_{j} + \sum_{i \in \mathcal{J}} 
		\left[1 - a_{i} \exp\left( \alpha_{i} - \beta_{i} x_{i} \right) \right] \rho_{ij} v_{i}
		& \forall \ j \in \mathcal{J} \nonumber \\
& a_{j} \ \ \in \ \ [0,1]
		& \forall \ j \in \mathcal{J} \nonumber \\
& x_{j} \ \ \leq \ \ \overline{x}_{j}
		& \forall \ j \in \mathcal{J} \nonumber \\
& x_{j} \ \ \geq \ \ \underline{x}_{j}
		& \forall \ j \in \mathcal{J} \nonumber 
\end{align}
\end{subequations}
We have: 

\begin{lemma}
\label{lem:sp2 integral}
If $(v^*, x^*, a^*)$ is an optimal solution to (\ref{eqn:sp2}), then $a^*$ is integral, and $(v^*, x^*, a^*)$ is also optimal to (\ref{eqn:sp1}). 
\end{lemma}

\begin{proof}
	See Appendix. 
\end{proof}

That is, we can solve (\ref{eqn:sp2}) by solving (\ref{eqn:sp1})

Next, we add a variable $d \defi (d_{j}, \ j \in \mathcal{J})$ to denote the choice probabilities of products. 
Thus, $d_{j} = a_{j} \exp\left( \alpha_{j} - \beta_{j} x_{j} \right) v_{j}$ for each $j \in \mathcal{J}$, and we obtain the following optimization problem: 
\begin{subequations}
\begin{align}
\max_{v, \ d, \ x, \ a} \quad
	& \sum_{j \in \mathcal{J}} \left(x_{j} - \psi_{j}\right) d_{j}
\tag{$\mathsf{SP}_2'$}
\label{eqn:sp2p} \\
\text{s.t.} \quad
& v_{j} \ \ = \ \ 
		\theta_{j} + \sum_{i \in \mathcal{J}} 
		\rho_{ij} (v_{i} - d_{i})
		& \forall \ j \in \mathcal{J} \nonumber \\
& d_{j} \ \ = \ \ a_{j} \exp\left( \alpha_{j} - \beta_{j} x_{j} \right) v_{j}
		& \forall \ j \in \mathcal{J} \nonumber \\
& a_{j} \ \ \in \ \ [0,1]
		& \forall \ j \in \mathcal{J} \nonumber \\
& x_{j} \ \ \leq \ \ \overline{x}_{j}
		& \forall \ j \in \mathcal{J} \nonumber \\
& x_{j} \ \ \geq \ \ \underline{x}_{j}
		& \forall \ j \in \mathcal{J} \nonumber 
\end{align}
\end{subequations}
Clearly, for any $(v,d,x,a)$ feasible for (\ref{eqn:sp2p}), it holds that $(v,x,a)$ is feasible for (\ref{eqn:sp2}) and has the same objective value. Conversely, for any $(v,x,a)$ feasible for (\ref{eqn:sp2}), it holds that $(v,d,x,a)$, with $d$ given by $d_{j} = a_{j} \exp\left( \alpha_{j} - \beta_{j} x_{j} \right) v_{j}$ for each $j \in \mathcal{J}$, is feasible for (\ref{eqn:sp2p}) and has the same objective value. Now, recall that the solution to (\ref{eqn:balance}) equations is non-negative. Thus, $v \geq 0$ in any feasible solution to (\ref{eqn:sp2p}), which means $\exp\left( \alpha_{j} - \beta_{j} x_{j} \right) v_{j} \geq 0$ for every $j \in \mathcal{J}$. Thus, (\ref{eqn:sp2p}) can be solved by solving the following optimization problem: 
\begin{subequations}
\begin{align}
\max_{v, \ d, \ x} \quad
	& \sum_{j \in \mathcal{J}} \left(x_{j} - \psi_{j}\right) d_{j}
\tag{$\mathsf{SP}_2''$}
\label{eqn:sp2pp} \\
\text{s.t.} \quad
& v_{j} \ \ = \ \ 
		\theta_{j} + \sum_{i \in \mathcal{J}} 
		\rho_{ij} (v_{i} - d_{i})
		& \forall \ j \in \mathcal{J} \label{eqn:sp2pp-balance} \\
& d_{j} \ \ \leq \ \ \exp\left( \alpha_{j} - \beta_{j} x_{j} \right) v_{j}
		& \forall \ j \in \mathcal{J} \label{eqn:sp2pp-d} \\
& d_{j} \ \ \geq \ \ 0
		& \forall \ j \in \mathcal{J} \label{eqn:sp2pp-d0} \\
& x_{j} \ \ \leq \ \ \overline{x}_{j}
		& \forall \ j \in \mathcal{J} \label{eqn:sp2pp-upper} \\
& x_{j} \ \ \geq \ \ \underline{x}_{j}
		& \forall \ j \in \mathcal{J} \label{eqn:sp2pp-lower} 
\end{align}
\end{subequations}
More concretely, for any $(v,d,x,a)$ feasible for (\ref{eqn:sp2p}), it holds that $(v,d,x)$ is feasible for (\ref{eqn:sp2pp}) and has the same objective value. Conversely, for any $(v,d,x)$ feasible for (\ref{eqn:sp2pp}), it holds that $(v,d,x,a)$, with $a$ given by 
\begin{align*}
a_{j} &= 
\begin{cases}
		0  & \text{if } d_{j} = 0 \\
		d_{j} / \exp\left( \alpha_{j} - \beta_{j} x_{j} \right) v_{j} & \text{if } d_{j} > 0 \\
\end{cases}
		& \forall \ j \in \mathcal{J} 
\end{align*}
is feasible for (\ref{eqn:sp2p}) and has the same objective value.

\begin{align*}
& d_{j} \ln\left( \frac{d_{j}}{v_{j}} \right) \ \ \leq \ \ \alpha_{j} d_{j} - \beta_{j} u_{j} 
		& \forall \ j \in \mathcal{J} 
\end{align*}

Next, we show that an optimal solution to (\ref{eqn:sp2pp}) can be obtained by solving a convex conic program: 
\begin{subequations}
\begin{align}
\max_{v, \ d, \ u} \quad
	& \sum_{j \in \mathcal{J}} \left(u_{j} - \psi_{j} d_{j} \right)
\tag{$\mathsf{SP}_3$}
\label{eqn:sp3} \\
\text{s.t.} \quad
& v_{j} \ \ = \ \ 
		\theta_{j} + \sum_{i \in \mathcal{J}} 
		\rho_{ij} (v_{i} - d_{i})
		& \forall \ j \in \mathcal{J} \label{eqn:sp3-balance} \\
&(v_{j}, \ d_{j}, \ \beta_{j} u_{j} - \alpha_{j} d_{j}) 
		\ \ \in \ \ \mathcal{K}_{\text{exp}}
		& \forall \ j \in \mathcal{J} \label{eqn:sp3-cone} \\
& u_{j} \ \ \leq \ \ \overline{x}_{j} d_{j} 
		& \forall \ j \in \mathcal{J} \label{eqn:sp3-upper} \\
& u_{j} \ \ \geq \ \ \underline{x}_{j} d_{j} 
		& \forall \ j \in \mathcal{J} \label{eqn:sp3-lower}
\end{align}
\end{subequations}
where
\begin{align*}
\mathcal{K}_{\exp} \ \ & \defi \ \ \mbox{closure}\big\{(a_{1}, a_{2}, a_{3}) \; : \; a_{3} \leq a_{2} \ln(a_{1} / a_{2}), \; a_{1} > 0, \; a_{2} > 0\big\} \\
& = \ \ \big\{(a_{1}, a_{2}, a_{3}) \; : \; a_{3} \leq a_{2} \ln(a_{1} / a_{2}), \; a_{1} > 0, \; a_{2} > 0\big\} \cup \big\{(a_{1}, 0, a_{3}) \; : \; a_{1} \geq 0, \; a_{3} \leq 0\big\}
\end{align*}
denotes the exponential cone. 
Consider any $(v,d,u)$ feasible for (\ref{eqn:sp3}), and let 
\begin{align*}
x_{j} &= 
\begin{cases}
		\underline{x}_{j} & \text{if } d_{j} = 0 \\
		u_{j} / d_{j} & \text{if } d_{j} > 0 \\
\end{cases}
		& \forall \ j \in \mathcal{J} 
\end{align*}
Then the conic constraint (\ref{eqn:sp3-cone}) implies that either
\begin{align*}
& \ln\left( \frac{d_{j}}{v_{j}} \right) \ \ \leq \ \ \alpha_{j} - \beta_{j} x_{j} 
		& \forall \ j \in \mathcal{J} 
\end{align*}
(when $d_{j} > 0$), or $d_{j} = 0$. 
This means constraints (\ref{eqn:sp2pp-d}) and (\ref{eqn:sp2pp-d0}) hold in (\ref{eqn:sp2pp}). 
Thus, $(v,d,x)$ is feasible for (\ref{eqn:sp2pp}) and has the same objective value. 
Conversely, consider any $(v,d,x)$ feasible for (\ref{eqn:sp2pp}). 
Then it holds that $(v,d,u)$, with $u$ given by $u_{j} = x_{j} d_{j}$ for each $j \in \mathcal{J}$, is feasible for (\ref{eqn:sp3}) and has the same objective value. 

Also, we have: 
\begin{lemma}
\label{lem:sp3 optimal}
(\ref{eqn:sp3}) has an optimal solution.
\end{lemma}

\begin{proof}
	See Appendix. 
\end{proof}

Next, we summarize the results for (\ref{eqn:sp1}): 
\begin{theorem}
\label{lem:sp1 optimal}
(\ref{eqn:sp1}) has an optimal solution, which can be found by solving the convex conic program (\ref{eqn:sp3}): Let $(v^*,d^*,u^*)$ be an optimal solution to (\ref{eqn:sp3}). Let 
\begin{align*}
x_{j}^* &= 
\begin{cases}
		\underline{x}_{j} & \text{if } d_{j}^* = 0 \\
		u_{j}^* / d_{j}^* & \text{if } d_{j}^* > 0 \\
\end{cases}
		& \forall \ j \in \mathcal{J} 
\end{align*}
and let
\begin{align*}
a_{j} &= 
\begin{cases}
		0  & \text{if } d_{j}^* = 0 \\
		d_{j}^* / \exp\left( \alpha_{j} - \beta_{j} x_{j}^* \right) v_{j}^* & \text{if } d_{j}^* > 0 \\
\end{cases}
		& \forall \ j \in \mathcal{J} 
\end{align*}
Then $(v^*,x^*,a^*)$ is an optimal solution to (\ref{eqn:sp1}). 
\end{theorem}

%------------------------------------------------------------------------
%	Section
%------------------------------------------------------------------------

\section{Network Revenue Management with Pricing}
\label{sec:network}

In this section, we extend the previous results to the network revenue management setting, where a seller manages a network of resources, each of which has a limited amount of capacity. 
Each product uses a particular combination of resources. 
Customers arrive at the system one by one and choose among the offered products according to a discrete choice model. 
When the seller sells a product, it generates revenue and consumes the capacities of the resources used by the product. 
The seller's objective is to maximize the total expected profit or revenue over a planning time period, with constraints in resource capacity. 

\cite{feldman2017revenue} formulated a linear programming approximation to the network revenue management problem under the Markov chain choice model. 
This section extends this network revenue management problem to allow pricing decisions. 
Let $\mathcal{R}$ denote the set of resources. 
For each resource $r \in \mathcal{R}$, let $b_{r}$ denote the amount of the resource available for use over the time period. 
For each pair of product $j \in \mathcal{J}$ and resource $r \in \mathcal{R}$, let $\phi_{rj}$ denote the amount of resource~$r$ needed per unit of product~$j$. 
Let $T$ be the length of the planning period. 
We assume that the customer arrival during $t \in [0,T]$ follows a Poisson distribution with rate $\lambda(t)$. 
Thus, we formulate the extended network revenue management problem as: 
\begin{subequations}
\begin{align}
\max_{\mathcal{A}(t), \ v(t), \ x(t)} \quad
	& \int_{t=0}^{T} \lambda(t)
			\left( \sum_{j \in \mathcal{A}(t)} \left(x_{j}(t) - \psi_{j}\right) 
			\exp\left( \alpha_{j} - \beta_{j} x_{j}(t) \right) v_{j}(t)  \right) \diff t
\tag{$\mathsf{CBP}_1$}
\label{eqn:cbp1} \\
\text{s.t.} \quad
& \int_{t=0}^{T} \lambda(t) \left(\sum_{j \in \mathcal{A}(t)} \phi_{rj} 
		\exp\left( \alpha_{j} - \beta_{j} x_{j}(t) \right) v_{j}(t) \right) \diff t
		\ \ \leq \ \ \varphi_{r}
		& \forall \ r \in \mathcal{R} \nonumber \\
& v_{j}(t)  \ \ = \ \ 
		\theta_{j} + \sum_{i \in \mathcal{J}} 
		\left[1 - \mathbf{1}(i \in \mathcal{A}(t)) \exp\left( \alpha_{i} - \beta_{i} x_{i}(t) \right) \right] 
		\rho_{ij} v_{i}(t) 
		& \forall \ j \in \mathcal{J} \ , \ t \in [0, T] \nonumber \\
& x_{j}(t) \ \ \leq \ \ \overline{x}_{j}
		& \forall \ j \in \mathcal{J} \ , \ t \in [0, T] \nonumber \\
& x_{j}(t) \ \ \geq \ \ \underline{x}_{j}
		& \forall \ j \in \mathcal{J} \ , \ t \in [0, T] \nonumber \\
& \mathcal{A}(t) \ \ \subseteq \ \ \mathcal{J}
		& \forall \ t \in [0, T] \nonumber 
\end{align}
\end{subequations}
where $v(t) \defi (v_{j}(t) : \ j \in \mathcal{J})$ and $x(t) \defi (x_{j}(t) : \ j \in \mathcal{J})$. 

Problem (\ref{eqn:cbp1}) is hard to work with for many reasons. In this section, we provide an efficient solution approach to (\ref{eqn:cbp1}) through the following steps: 
\begin{enumerate}
	\item We show that if (\ref{eqn:cbp1}) has an optimal solution, then it has an optimal solution with a constant price vector (i.e. $x(t)$ is a constant over $t \in [0, 1]$). Since $\mathcal{A}(t)$ can only take a finite number of values, this means we can reduce (\ref{eqn:cbp1}) to a problem (\ref{eqn:cbp2}) with finite number of decision variables which do not depend on time. 
	\item The number of variables in (\ref{eqn:cbp2}) grows exponentially as the number of products grow. Thus, (\ref{eqn:cbp2}) is still hard to solve. We then show (\ref{eqn:cbp2}) can be solved by solving a much smaller problem (\ref{eqn:rp1}) (i.e. an optimal solution to (\ref{eqn:rp1}) can be transformed into an optimal solution to (\ref{eqn:cbp2}) in polynomial time). 
	\item We show that (\ref{eqn:rp1}) can be solved by solving a convex conic program (\ref{eqn:rp2}) (i.e. an optimal solution to (\ref{eqn:rp2}) can be transformed into an optimal solution to (\ref{eqn:rp1}) in polynomial time). 
\end{enumerate}
Thus, we can solve (\ref{eqn:cbp1}) efficiently by solving a convex conic program (\ref{eqn:rp2}), and then transform the optimal solution to (\ref{eqn:rp2}) we found back to an optimal solution to (\ref{eqn:cbp1}) efficiently.

%------------------------------------------------------------------------

\subsection{Constant Price Vector}

Suppose that the seller chooses the prices of products at the beginning of the time period then decides which products to make available dynamically during the time period. 
Let $w^{\mathcal{A}}$ denote the proportion of customers that are offered assortment $\mathcal{A} \subseteq \mathcal{J}$ during the time period, and let $w \defi (w^{\mathcal{A}}, \mathcal{A} \subseteq \mathcal{J})$. 
Let $\overline{\lambda} = \int_{t=0}^{T} \lambda(t) \diff t$. 
Consider the following choice-based optimization problem: 
\begin{subequations}
\begin{align}
\max_{w, \ \mathbf{v}, \ x} \quad
	& \sum_{\mathcal{A} \subseteq \mathcal{J}} \overline{\lambda} w^{\mathcal{A}} 
			\left( \sum_{j \in \mathcal{A}} \left(x_{j} - \psi_{j}\right) 
			\exp\left( \alpha_{j} - \beta_{j} x_{j} \right) v_{j}^{\mathcal{A}}  \right)
\tag{$\mathsf{CBP}_2$}
\label{eqn:cbp2} \\
\text{s.t.} \quad
& \sum_{\mathcal{A} \subseteq \mathcal{J}} \overline{\lambda} w^{\mathcal{A}} 
		\left(\sum_{j \in \mathcal{A}} \phi_{rj} \exp\left( \alpha_{j} - \beta_{j} x_{j} \right) v_{j}^{\mathcal{A}} \right)
		\ \ \leq \ \ \varphi_{r}
		& \forall \ r \in \mathcal{R} \nonumber \\
& \sum_{\mathcal{A} \subseteq \mathcal{J}} w^{\mathcal{A}} \ \ = \ \ 1 \nonumber \\
& w \ \ \geq \ \ 0 \nonumber \\
& v_{j}^{\mathcal{A}}  \ \ = \ \ 
		\theta_{j} + \sum_{i \in \mathcal{J}} 
		\left[1 - \mathbf{1}(i \in \mathcal{A}) \exp\left( \alpha_{i} - \beta_{i} x_{i} \right) \right] 
		\rho_{ij} v_{i}^{\mathcal{A}} 
		& \forall \ j \in \mathcal{J}  \ , \ \mathcal{A} \subseteq \mathcal{J} \nonumber \\
& x_{j} \ \ \leq \ \ \overline{x}_{j}
		& \forall \ j \in \mathcal{J} \nonumber \\
& x_{j} \ \ \geq \ \ \underline{x}_{j}
		& \forall \ j \in \mathcal{J} \nonumber 
\end{align}
\end{subequations}
where $\mathbf{v} \defi (v_{j}^{\mathcal{A}}, \ j \in \mathcal{J} , \ \mathcal{A} \subseteq \mathcal{J})$. 

In (\ref{eqn:cbp1}), prices of products can change at any time during the planning period. In (\ref{eqn:cbp2}), however, prices are assumed to stay constant over time. 
Clearly, (\ref{eqn:cbp1}) is a relaxation to (\ref{eqn:cbp2}), and any feasible solution to (\ref{eqn:cbp2}) can be transformed into a feasible solution to (\ref{eqn:cbp1}) with equal objective value. 
However, it is not trivial to ask whether allowing prices to change over time improves revenue. 
Here, we have: 

\begin{theorem}
	\label{thm:rm-constant-price}
	The optimal objective value of (\ref{eqn:cbp1}) does not exceed that of (\ref{eqn:cbp2}). 
\end{theorem}

\begin{proof}
	See Appendix. Note that the proof needs to use properties of (\ref{eqn:cbp2}) that will be shown later in this section. 
\end{proof}

Thus, an optimal solution of (\ref{eqn:cbp2}) can be transformed into optimal solution of (\ref{eqn:cbp1}). 
In other words, allowing prices to change over time does not improve revenue, and it is enough to use a constant price vector throughout the time period.

%------------------------------------------------------------------------

\subsection{The Reduced Problem}

As we have pointed out, (\ref{eqn:cbp2}) is hard to solve. 
The number of variables and constraints in it grows exponentially as the size of $\mathcal{J}$ grows. 
Next, we show that we can find an optimal solution to (\ref{eqn:cbp2}) by solving a much smaller optimization problem: 
\begin{subequations}
\begin{align}
\max_{v, \ x, \ a} \quad
	& \overline{\lambda} \sum_{j \in \mathcal{J}} \left(x_{j} - \psi_{j}\right) 
	a_{j} \exp\left( \alpha_{j} - \beta_{j} x_{j} \right) v_{j}
\tag{$\mathsf{RP}_1$}
\label{eqn:rp1} \\
\text{s.t.} \quad
& v_{j} \ \ = \ \ 
		\theta_{j} + \sum_{i \in \mathcal{J}} 
		\left[1 - a_{i} \exp\left( \alpha_{i} - \beta_{i} x_{i} \right) \right] \rho_{ij} v_{i}
		& \forall \ j \in \mathcal{J} \nonumber \\
& \overline{\lambda} \left(\sum_{j \in \mathcal{J}} \phi_{rj} 
		a_{j} \exp\left( \alpha_{j} - \beta_{j} x_{j} \right) v_{j} \right)
		\ \ \leq \ \ \varphi_{r}
		& \forall \ r \in \mathcal{R} \nonumber \\
& a_{j} \ \ \in \ \ [0,1]
		& \forall \ j \in \mathcal{J} \nonumber \\
& x_{j} \ \ \leq \ \ \overline{x}_{j}
		& \forall \ j \in \mathcal{J} \nonumber \\
& x_{j} \ \ \geq \ \ \underline{x}_{j}
		& \forall \ j \in \mathcal{J} \nonumber 
\end{align}
\end{subequations}
where $a \defi (a_{j}, \ j \in \mathcal{J})$. 

Our approach here is very similar to that from \cite{feldman2017revenue}, in which the authors showed that the size of their choice-based linear program can be reduced drastically. 
We adjusted their approach such that it fits the extended problem we formulated. 
First, consider any feasible solution $(\mathbf{v}, x, w)$ to (\ref{eqn:cbp2}). Let 
\begin{align*}
v_{j} &= \sum_{\mathcal{A} \subseteq \mathcal{J}} w^{\mathcal{A}} v_{j}^{\mathcal{A}}
		\quad , \quad a_{j} = \sum_{\mathcal{A} \subseteq \mathcal{J}} 
				\mathbf{1}(j \in \mathcal{J}) w^{\mathcal{A}} v_{j}^{\mathcal{A}}
				 \Big/ \sum_{\mathcal{A} \subseteq \mathcal{J}} w^{\mathcal{A}} v_{j}^{\mathcal{A}}
		& \forall \ j \in \mathcal{J} 
\end{align*}
Then $(v, x, a)$ is clearly a feasible solution to (\ref{eqn:rp1}) with an identical objective value as $(\mathbf{v}, x, w)$ in (\ref{eqn:cbp2}). Second, consider the algorithm: 

\hfill

\begin{algorithm}[H]
\caption{Dimension Reduction}
\label{algo:reduce}
\SetAlgoLined
\textbf{Initialization} : Set $k=1$. 
Set $\hat{v}_{j}^{(1)} = v_{j}$, $\hat{a}_{j}^{(1)} = a_{j}$ for every $j \in \mathcal{J}$

\textbf{Step 1} : Set $\mathcal{A}_{k} = \{j \in \mathcal{J} : \hat{a}_{j}^{(k)} \hat{v}_{j}^{(k)} > 0\}$. 
If $\mathcal{A}_{k} = \varnothing$, then set $y^{\mathcal{A}_{k}} = 1$ and stop. 

\textbf{Step 2} : Compute the unique (see Remark \ref{rmk:balance-unique}) solution $(v_{j}^{\mathcal{A}_{k}} : j \in \mathcal{J})$ to the system
\begin{align*}
v_{j}^{\mathcal{A}_{k}}  &\ \ = \ \ 
		\theta_{j} + \sum_{i \in \mathcal{J}} 
		\left[1 - \mathbf{1}(i \in \mathcal{A}_{k}) \exp\left( \alpha_{i} - \beta_{i} x_{i} \right) \right] 
		\rho_{ij} v_{i}^{\mathcal{A}_{k}} 
		& \forall \ j \in \mathcal{J}
\end{align*}

\textbf{Step 3} : Set $y^{\mathcal{A}_{k}} = \min\{\hat{a}_{j}^{(k)} \hat{v}_{j}^{(k)} / v_{j}^{\mathcal{A}_{k}} : j \in \mathcal{A}_{k}\}$. If $y^{\mathcal{A}_{k}} = 1$, then stop. 

\textbf{Step 4} : Set
\begin{align*}
\hat{v}_{j}^{(k+1)} &= (\hat{v}_{j}^{(k)} - y^{\mathcal{A}_{k}} v_{j}^{\mathcal{A}_{k}}) 
		/ (1 - y^{\mathcal{A}_{k}}) 
		& \forall \ j \in \mathcal{J} \\
\hat{a}_{j}^{(k+1)} &= (\hat{a}_{j}^{(k)} \hat{v}_{j}^{(k)} - y^{\mathcal{A}_{k}} v_{j}^{\mathcal{A}_{k}}) 
				/ ((1 - y^{\mathcal{A}_{k}}) \hat{v}_{j}^{(k+1)}) 
		& \forall \ j \in \mathcal{A}_{k} \\
\hat{a}_{j}^{(k+1)} &= 0 
		& \forall \ j \in \mathcal{J} \setminus \mathcal{A}_{k} 
\end{align*}

\textbf{Step 5} : Increase $k$ by $1$, and go to Step 1. 
\end{algorithm}

\hfill

We have: 

\begin{lemma}
\label{lem:rp1-to-cbp2}
Let $(v, x, a)$ be a feasible solution to (\ref{eqn:rp1}). 
Using $(v, x, a)$ as the input, Algorithm \ref{algo:reduce} stops after at most $J+1$ iterations. 
Let $(\mathcal{A}_{k} : k = 1, \dots, K)$, $(v^{\mathcal{A}_{k}} : k = 1, \dots, K)$ and $(y^{\mathcal{A}_{k}} : k = 1, \dots, K)$ denote the output of Algorithm \ref{algo:reduce} using $(v, x, a)$ as the input, where $K$ is the number of iterations after which the algorithm stops. 
In addition, let 
\begin{align*}
w^{\mathcal{A}_{k}} &= (1 - y^{\mathcal{A}_1}) \cdots (1 - y^{\mathcal{A}_{k-1}}) y^{\mathcal{A}_{k}}
		\quad , \quad k = 1, \dots, K 
\end{align*}
let $w^{\mathcal{A}} = 0$ for any $\mathcal{A} \subseteq \mathcal{J}$ such that $\mathcal{A} \notin \{\mathcal{A}_1, \dots, \mathcal{A}_{k}\}$, and let $v^{\mathcal{A}}$ be the unique (see Remark \ref{rmk:balance-unique}) solution of the system
\begin{align*}
v_{j}^{\mathcal{A}}  
&\ \ = \ \ \theta_{j} + \sum_{i \in \mathcal{J}} 
		\left[1 - \mathbf{1}(i \in \mathcal{A}) \exp\left( \alpha_{i} - \beta_{i} x_{i} \right) \right] 
		\rho_{ij} v_{i}^{\mathcal{A}} 
		& \forall \ j \in \mathcal{J} 
\end{align*}
for every $\mathcal{A} \subseteq \mathcal{J}$ given the value of $x$. 
Then $(\mathbf{v}, x, w)$ is a feasible solution to (\ref{eqn:cbp2}) with an identical objective value as $(v, x, a)$ in (\ref{eqn:rp1}). 
\end{lemma}

Next, we summarize the results for (\ref{eqn:rp1}): 

\begin{theorem}
\label{thm:rp1-to-cbp2}
Let $(v^*, x^*, a^*)$ be an optimal solution to (\ref{eqn:rp1}). Let $(\mathbf{v}^*, x^*, w^*)$ be a feasible solution to (\ref{eqn:cbp2}) constructed as in Lemma \ref{lem:rp1-to-cbp2}. Then $(\mathbf{v}^*, x^*, w^*)$ is optimal to (\ref{eqn:cbp2}). On the other hand, let $(\mathbf{v}^*, x^*, w^*)$ be an optimal solution to (\ref{eqn:cbp2}), and let 
\begin{align*}
(v^*)_{j}^{\mathcal{A}} &= \sum_{\mathcal{A} \subseteq \mathcal{J}} (w^*)^{\mathcal{A}} (v^*)_{j}^{\mathcal{A}}
		\quad , \quad a_{j}^* = \sum_{\mathcal{A} \subseteq \mathcal{J}} 
				\mathbf{1}(j \in \mathcal{J}) (w^*)^{\mathcal{A}} (v^*)_{j}^{\mathcal{A}}
				 \Big/ \sum_{\mathcal{A} \subseteq \mathcal{J}} (w^*)^{\mathcal{A}} (v^*)_{j}^{\mathcal{A}}
		& \forall \ j \in \mathcal{J} 
\end{align*}
Then $(v^*, x^*, a^*)$ is an optimal solution to (\ref{eqn:rp1}). 
\end{theorem}

\begin{remark}
\textbf{Scenario of Constant Assortment:} Consider the scenario when there is no resource constraint in (\ref{eqn:cbp2}). Clearly, its reduced problem (\ref{eqn:rp1}) will contain no resource constraint as well. Such a problem (\ref{eqn:rp1}) only differs from (\ref{eqn:sp2}) by a scalar $\overline{\lambda}$. Recall that in Section \ref{sec:joint}, we have shown that there must be an optimal solution to (\ref{eqn:sp2}) such that $a$ is integral. Thus, there must be an optimal solution with integral $a$ to (\ref{eqn:rp1}) when there is no resource constraint. By Algorithm \ref{algo:reduce}, such a solution translates to an optimal solution of (\ref{eqn:cbp2}) where $w^{\mathcal{A}}=1$ for one assortment $\mathcal{A} \in \mathcal{J}$, while $w^{\mathcal{A}'}=0$ for all other assortments $\mathcal{A}' \in \mathcal{J}$. In other words, when (\ref{eqn:cbp2}) has no resource constraint, it has an optimal solution with a constant assortment. 
\end{remark}

%------------------------------------------------------------------------

\subsection{Convex Conic Transformation}

Now, we have shown that (\ref{eqn:cbp2}) can be solved by solving (\ref{eqn:rp1}), which has a much smaller number of variables and constraints. It left to show that (\ref{eqn:rp1}) can be solved efficiently. We show that (\ref{eqn:rp1}) can be solved by solving the following convex conic optimization problem: 
\begin{subequations}
\begin{align}
\max_{v, \ d, \ u} \quad
	& \overline{\lambda} \sum_{j \in \mathcal{J}} \left(u_{j} - \psi_{j} d_{j} \right)
\tag{$\mathsf{RP}_2$}
\label{eqn:rp2} \\
\text{s.t.} \quad
& v_{j} \ \ = \ \ 
		\theta_{j} + \sum_{i \in \mathcal{J}} 
		\rho_{ij} (v_{i} - d_{i})
		& \forall \ j \in \mathcal{J} \label{eqn:rp2-balance} \\
& \overline{\lambda} \left(\sum_{j \in \mathcal{J}} \phi_{rj} d_{j} \right)
		\ \ \leq \ \ \varphi_{r}
		& \forall \ r \in \mathcal{R} \\
&(v_{j}, \ d_{j}, \ \beta_{j} u_{j} - \alpha_{j} d_{j}) 
		\ \ \in \ \ \mathcal{K}_{\text{exp}}
		& \forall \ j \in \mathcal{J} \label{eqn:rp2-cone} \\
& u_{j} \ \ \leq \ \ \overline{x}_{j} d_{j} 
		& \forall \ j \in \mathcal{J} \label{eqn:rp2-upper} \\
& u_{j} \ \ \geq \ \ \underline{x}_{j} d_{j} 
		& \forall \ j \in \mathcal{J} \label{eqn:rp2-lower}
\end{align}
\end{subequations}

More concretely, we have: 
\begin{lemma}
\label{lem:rp2 optimal}
If (\ref{eqn:rp1}) is feasible, then (\ref{eqn:rp2}) has an optimal solution.
\end{lemma}

\begin{proof}
	See Appendix. 
\end{proof}

and
\begin{theorem}
\label{lem:cbp2 optimal}
If (\ref{eqn:rp1}) is feasible, then it has an optimal solution, which can be found by solving the convex conic program (\ref{eqn:rp2}): Let $(v^*,d^*,u^*)$ be an optimal solution to (\ref{eqn:rp2}). Let 
\begin{align*}
x_{j}^* \ \ &= \ \ 
\begin{cases}
		\underline{x}_{j} & \text{if } d_{j}^* = 0 \\
		u_{j}^* / d_{j}^* & \text{if } d_{j}^* > 0 \\
\end{cases}
		& \forall \ j \in \mathcal{J} 
\end{align*}
and let
\begin{align*}
a_{j} \ \ &= \ \ 
\begin{cases}
		0  & \text{if } d_{j}^* = 0 \\
		d_{j}^* / \exp\left( \alpha_{j} - \beta_{j} x_{j}^* \right) v_{j}^* & \text{if } d_{j}^* > 0 \\
\end{cases}
		& \forall \ j \in \mathcal{J} 
\end{align*}
Then $(v^*,x^*,a^*)$ is an optimal solution to (\ref{eqn:rp1}). 
\end{theorem}

\begin{proof}
	Theorem \ref{lem:cbp2 optimal} can be shown through the identical steps we used to show Theorem \ref{lem:sp1 optimal}. The only differences between (\ref{eqn:rp1}) and (\ref{eqn:sp2}) (and between (\ref{eqn:rp2}) and (\ref{eqn:sp3})) are the scaling parameter $\overline{\lambda}$ and the resource constraints. None of the changes affect the transformation we applied to (\ref{eqn:sp2}) for obtaining (\ref{eqn:sp3})), and the proof steps in Section \ref{sec:joint} still hold. 
\end{proof}

%------------------------------------------------------------------------

\subsection{Convert to Special Cases}
\label{sec:sub:special}

The revenue management problem in \cite{feldman2017revenue} (with exogenous prices) can be solved as a special case of (\ref{eqn:cbp2}). 
To do this, we can simply set both $\overline{x}_{j}$ and $\underline{x}_{j}$ equal to the target price of product $j$ for every $j \in \mathcal{J}$ in (\ref{eqn:cbp2}). 
This change will be carried through (\ref{eqn:rp1}) and (\ref{eqn:rp2}). 
The tractability results we showed in this section hold automatically. 
Thus, the optimal solution we obtain for (\ref{eqn:cbp2}) will satisfy $x_{j} = \overline{x}_{j} = \underline{x}_{j}$. 

Next, we discuss how to solve the pricing problem in \cite{dong2019pricing} (with exogenous assortment) as special cases of problem (\ref{eqn:cbp2}). 
Without loss of generality, we assume that the offer set is $\mathcal{J}$ itself throughout time. 
Thus, we are solving (\ref{eqn:cbp2}) with an additional constraint $w^{\mathcal{J}} = 1$. 
In short, this can be done by first solving 
\begin{subequations}
\begin{align}
\max_{v, \ d, \ u} \quad
	& \overline{\lambda} \sum_{j \in \mathcal{J}} \left(u_{j} - \psi_{j} d_{j} \right)
\tag{$\mathsf{RP}_3$}
\label{eqn:rp3} \\
\text{s.t.} \quad
& v_{j} \ \ = \ \ 
		\theta_{j} + \sum_{i \in \mathcal{J}} 
		\rho_{ij} (v_{i} - d_{i})
		& \forall \ j \in \mathcal{J} \label{eqn:rp3-balance} \\
& \overline{\lambda} \left(\sum_{j \in \mathcal{J}} \phi_{rj} d_{j} \right)
		\ \ \leq \ \ \varphi_{r}
		& \forall \ r \in \mathcal{R} \\
&(v_{j}, \ d_{j}, \ \beta_{j} u_{j} - \alpha_{j} d_{j}) 
		\ \ \in \ \ \mathcal{K}_{\text{exp}}
		& \forall \ j \in \mathcal{J} \label{eqn:rp3-cone} \\
&(d_{j}, \ v_{j}, \ \alpha_{j} v_{j} - \beta_{j} \overline{x}_{j} v_{j}) 
		\ \ \in \ \ \mathcal{K}_{\text{exp}}
		& \forall \ j \in \mathcal{J} \label{eqn:rp3-cone-dummy} \\
& u_{j} \ \ \leq \ \ \overline{x}_{j} d_{j} 
		& \forall \ j \in \mathcal{J} \label{eqn:rp3-upper} \\
& u_{j} \ \ \geq \ \ \underline{x}_{j} d_{j} 
		& \forall \ j \in \mathcal{J} \label{eqn:rp3-lower}
\end{align}
\end{subequations}
instead of (\ref{eqn:rp2}), and then follow the steps in this section backwards to get the solutions to (\ref{eqn:rp1}) and (\ref{eqn:cbp2}). 
The details and proofs for approach can be found in \cite{shao2020tractable}. 

Note that the only difference between (\ref{eqn:rp2}) and (\ref{eqn:rp3}) is the additional conic constraint (\ref{eqn:rp3-cone-dummy}). 
Intuitively, increasing $u_{j}$ in (\ref{eqn:rp3}) gives us a better solution, and the only constraints that put upper bounds on $u_{j}$ are (\ref{eqn:rp3-cone}) and (\ref{eqn:rp3-upper}). 
Thus, one of the two constraints must hold tight at optimality. 
Now, without adding unnecessary restrictions, the conic constraint (\ref{eqn:rp3-cone-dummy}) gaurantees that if (\ref{eqn:rp3-cone}) is not tight, then (\ref{eqn:rp3-upper}) must not be tight.  
Thus, any optimal solution to (\ref{eqn:rp3}) must be on the boundary of (\ref{eqn:rp3-cone}). 
As a result, through Theorem \ref{lem:cbp2 optimal}, the solution we obtain to (\ref{eqn:rp1}) will have $a_{j} = 1 \ , \ j \in \mathcal{J}$. (In fact, by solving (\ref{eqn:rp3}), we are solving a modifed (\ref{eqn:rp1}) with an additional constraint $a_{j} = 1 \ , \ j \in \mathcal{J}$.)
Similarly, the solution we obtain to (\ref{eqn:cbp2}) will have $w^{\mathcal{J}} = 1$.

\section{Conclusion}
\label{sec:conclusion}

In this paper, we formulated the revenue management problem (\ref{eqn:cbp1}) under the Markov chain choice model with joint price and assortment decisions. We showed that (\ref{eqn:cbp1}) can be solved by solving a small convex conic program (\ref{eqn:rp2}). A practical result on (\ref{eqn:cbp1}) is that allowing prices to change over time does not improve profit. Thus, it is enough to use a constant price vector throughout the planning period. Special cases of (\ref{eqn:cbp1}), including (1) the static joint optimization problem (\ref{eqn:sp1}), (2) the ``pricing only'' problem, and (3) the ``assortment only'' problem (both are covered in \ref{sec:sub:special}), are discussed in this paper. The concavity results on (\ref{eqn:cbp1}) extend to the special cases as well.

\newpage

%%%%%%%%%%%%%%%%%%%%%%%%%%%%%%
%%%%%%%%%%%%%%%%%%%%%%%%%%%%%%

% ------------------------------------------------------------
% REFERENCE
% ------------------------------------------------------------

\bibliography{library}

\begin{thebibliography}{28}
\providecommand{\natexlab}[1]{#1}
\providecommand{\url}[1]{\texttt{#1}}
\expandafter\ifx\csname urlstyle\endcsname\relax
  \providecommand{\doi}[1]{doi: #1}\else
  \providecommand{\doi}{doi: \begingroup \urlstyle{rm}\Url}\fi

\bibitem[Akcay et~al.(2010)Akcay, Natarajan, and Xu]{akcay2010joint}
Yalcin Akcay, Harihara~Prasad Natarajan, and Susan~H Xu.
\newblock {Joint Dynamic Pricing of Multiple Perishable Products under Consumer
  Choice}.
\newblock \emph{Management Science}, 56\penalty0 (8):\penalty0 1345--1361,
  2010.

\bibitem[Anderson and {De Palma}(1992)]{anderson1992multiproduct}
Simon~P Anderson and Andr{\'{e}} {De Palma}.
\newblock {Multiproduct Firms: A Nested Logit Approach}.
\newblock \emph{The Journal of Industrial Economics}, pages 261--276, 1992.

\bibitem[Aydin and Porteus(2008)]{aydin2008joint}
Goker Aydin and Evan~L Porteus.
\newblock {Joint Inventory and Pricing Decisions for an Assortment}.
\newblock \emph{Operations Research}, 56\penalty0 (5):\penalty0 1247--1255,
  2008.

\bibitem[Aydin and Ryan(2000)]{aydin2000product}
Goker Aydin and Jennifer~K Ryan.
\newblock {Product Line Selection and Pricing under the Multinomial Logit
  Choice Model}.
\newblock In \emph{Proceedings of the 2000 MSOM conference}. Citeseer, 2000.

\bibitem[Ben-Tal and Nemirovski(2001)]{ben2001lectures}
Aharon Ben-Tal and Arkadi Nemirovski.
\newblock \emph{{Lectures on Modern Convex Optimization: Analysis, Algorithms,
  and Engineering Applications}}.
\newblock SIAM, 2001.

\bibitem[Besanko et~al.(1998)Besanko, Gupta, and Jain]{besanko1998logit}
David Besanko, Sachin Gupta, and Dipak Jain.
\newblock {Logit Demand Estimation under Competitive Pricing Behavior: An
  Equilibrium Framework}.
\newblock \emph{Management Science}, 44\penalty0 (11-part-1):\penalty0
  1533--1547, 1998.

\bibitem[Blanchet et~al.(2016)Blanchet, Gallego, and Goyal]{blanchet2016markov}
Jose Blanchet, Guillermo Gallego, and Vineet Goyal.
\newblock {A Markov Chain Approximation to Choice Modeling}.
\newblock \emph{Operations Research}, 64\penalty0 (4):\penalty0 886--905, 2016.

\bibitem[Bront et~al.(2009)Bront, M{\'{e}}ndez-D$\backslash$'$\backslash$iaz,
  and Vulcano]{bront2009column}
Juan Jos{\'{e}}~Miranda Bront, Isabel
  M{\'{e}}ndez-D$\backslash$'$\backslash$iaz, and Gustavo Vulcano.
\newblock {A Column Generation Algorithm for Choice-based Network Revenue
  Management}.
\newblock \emph{Operations research}, 57\penalty0 (3):\penalty0 769--784, 2009.

\bibitem[Davis et~al.(2013)Davis, Gallego, and Topaloglu]{davis2013assortment}
James Davis, Guillermo Gallego, and Huseyin Topaloglu.
\newblock {Assortment Planning under the Multinomial Logit Model with Totally
  Unimodular Constraint Structures}.
\newblock \emph{Work in Progress}, 2013.

\bibitem[Davis et~al.(2017)Davis, Topaloglu, and Williamson]{davis2017pricing}
James~M Davis, Huseyin Topaloglu, and David~P Williamson.
\newblock {Pricing Problems under the Nested Logit Model with a Quality
  Consistency Constraint}.
\newblock \emph{INFORMS Journal on Computing}, 29\penalty0 (1):\penalty0
  54--76, 2017.

\bibitem[Dong et~al.(2019)Dong, Simsek, and Topaloglu]{dong2019pricing}
James Dong, A~Serdar Simsek, and Huseyin Topaloglu.
\newblock {Pricing Problems under the Markov Chain Choice Model}.
\newblock \emph{Production and Operations Management}, 28\penalty0
  (1):\penalty0 157--175, 2019.

\bibitem[Dong et~al.(2009)Dong, Kouvelis, and Tian]{dong2009dynamic}
Lingxiu Dong, Panos Kouvelis, and Zhongjun Tian.
\newblock {Dynamic Pricing and Inventory Control of Substitute Products}.
\newblock \emph{Manufacturing {\&} Service Operations Management}, 11\penalty0
  (2):\penalty0 317--339, 2009.

\bibitem[Feldman and Topaloglu(2017)]{feldman2017revenue}
Jacob~B Feldman and Huseyin Topaloglu.
\newblock {Revenue Management under the Markov Chain Choice Model}.
\newblock \emph{Operations Research}, 65\penalty0 (5):\penalty0 1322--1342,
  2017.

\bibitem[Gallego and Topaloglu(2014)]{gallego2014constrained}
Guillermo Gallego and Huseyin Topaloglu.
\newblock {Constrained Assortment Optimization for the Nested Logit Model}.
\newblock \emph{Management Science}, 60\penalty0 (10):\penalty0 2583--2601,
  2014.

\bibitem[Gallego and Wang(2014)]{gallego2014multiproduct}
Guillermo Gallego and Ruxian Wang.
\newblock {Multiproduct Price Optimization and Competition under the Nested
  Logit Model with Product Differentiated Price Sensitivities}.
\newblock \emph{Operations Research}, 62\penalty0 (2):\penalty0 450--461, 2014.

\bibitem[Gallego et~al.(2004)Gallego, Iyengar, Phillips, and
  Dubey]{gallego2004managing}
Guillermo Gallego, Garud Iyengar, Robert Phillips, and Abhay Dubey.
\newblock {Managing Flexible Products on a Network}.
\newblock \emph{Available at SSRN 3567371}, 2004.

\bibitem[Gallego et~al.(2011)Gallego, Ratliff, and
  Shebalov]{gallego2011general}
Guillermo Gallego, Richard Ratliff, and Sergey Shebalov.
\newblock {A General Attraction Model and an Efficient Formulation for the
  Network Revenue Management Problem}.
\newblock \emph{Techinal Report, Columbia University, New York, NY}, 2011.

\bibitem[Hanson and Martin(1996)]{hanson1996optimizing}
Ward Hanson and Kipp Martin.
\newblock {Optimizing Multinomial Logit Profit Functions}.
\newblock \emph{Management Science}, 42\penalty0 (7):\penalty0 992--1003, 1996.

\bibitem[Keller et~al.(2014)Keller, Levi, and Perakis]{keller2014efficient}
Philipp~W Keller, Retsef Levi, and Georgia Perakis.
\newblock {Efficient Formulations for Pricing Under Attraction Demand Models}.
\newblock \emph{Mathematical Programming}, 145\penalty0 (1-2):\penalty0
  223--261, 2014.

\bibitem[Li et~al.(2015)Li, Rusmevichientong, and Topaloglu]{li2015d}
Guang Li, Paat Rusmevichientong, and Huseyin Topaloglu.
\newblock {The d-level nested logit model: Assortment and price optimization
  problems}.
\newblock \emph{Operations Research}, 63\penalty0 (2):\penalty0 325--342, 2015.

\bibitem[Li and Huh(2011)]{li2011pricing}
Hongmin Li and Woonghee~Tim Huh.
\newblock {Pricing Multiple Products with the Multinomial Logit and Nested
  Logit Models: Concavity and Implications}.
\newblock \emph{Manufacturing {\&} Service Operations Management}, 13\penalty0
  (4):\penalty0 549--563, 2011.

\bibitem[Liu and {Van Ryzin}(2008)]{liu2008choice}
Qian Liu and Garrett {Van Ryzin}.
\newblock {On the Choice-based Linear Programming Model for Network Revenue
  Management}.
\newblock \emph{Manufacturing {\&} Service Operations Management}, 10\penalty0
  (2):\penalty0 288--310, 2008.

\bibitem[Shao and Kleywegt(2020)]{shao2020tractable}
Hongzhang Shao and Anton~J Kleywegt.
\newblock {Tractable Constrained Optimization over Multiple Product Attributes
  under Discrete Choice Models}.
\newblock \emph{arXiv preprint arXiv:2007.09193}, 2020.

\bibitem[Song and Xue(2007)]{song2007demand}
Jing-Sheng Song and Zhengliang Xue.
\newblock {Demand Management and Inventory Control for Substitutable Products}.
\newblock \emph{Work in Progress}, 2007.

\bibitem[Strauss et~al.(2018)Strauss, Klein, and Steinhardt]{strauss2018review}
Arne~K Strauss, Robert Klein, and Claudius Steinhardt.
\newblock {A review of choice-based revenue management: Theory and methods}.
\newblock \emph{European journal of operational research}, 271\penalty0
  (2):\penalty0 375--387, 2018.

\bibitem[Wang(2012)]{wang2012joint}
Ruxian Wang.
\newblock {Joint Optimization of Assortment Selection and Pricing under the
  Capacitated Multinomial Logit Choice Model with Product-Differentiated Price
  Sensitivities}.
\newblock Technical report, Working paper, 2012.

\bibitem[Wang and Shen(2017)]{wang2017joint}
Yanqiao Wang and Zuo-Jun~Max Shen.
\newblock {Joint optimization of capacitated assortment and pricing problem
  under the tree logit model}.
\newblock Technical report, Technical report, University of California,
  Berkeley, CA, 2017.

\bibitem[Zhang et~al.(2018)Zhang, Rusmevichientong, and
  Topaloglu]{zhang2018multiproduct}
Heng Zhang, Paat Rusmevichientong, and Huseyin Topaloglu.
\newblock {Multiproduct Pricing under the Generalized Extreme Value Models with
  Homogeneous Price Sensitivity Parameters}.
\newblock \emph{Operations Research}, 66\penalty0 (6):\penalty0 1559--1570,
  2018.

\end{thebibliography}
\bibliographystyle{plainnat}

%%%%%%%%%%%%%%%%%%%%%%%%%%%%%%
%%%%%%%%%%%%%%%%%%%%%%%%%%%%%%

 \begin{appendix}

 \newpage
 
%--------------------------------------------------
%	Section
%--------------------------------------------------

\section{Proofs}

\subsection{Proof of Lemma \ref{lem:sp2 integral}}

{\color{purple}
\begin{manuallemma}{\ref{lem:sp2 integral}}
If $(v^*, x^*, a^*)$ is an optimal solution to (\ref{eqn:sp2}), then $a^*$ is integral, and $(v^*, x^*, a^*)$ is also optimal to (\ref{eqn:sp1}). 
\end{manuallemma}
}

\begin{proof}
(\ref{eqn:sp2}) is equivalent to
\begin{subequations}
\begin{align}
\max_{x} \quad
	& G(x)
\tag{$\mathsf{OP}$}
\label{eqn:inp1} \\
\text{s.t.} \quad
& x_{j} \ \ \leq \ \ \overline{x}_{j}
		& \forall \ j \in \mathcal{J} \nonumber \\
& x_{j} \ \ \geq \ \ \underline{x}_{j}
		& \forall \ j \in \mathcal{J} \nonumber 
\end{align}
\end{subequations}
where 
\begin{subequations}
\begin{align}
G(x) \ \ \defi \ \ 
\max_{v, \ a} \quad
	& \sum_{j \in \mathcal{J}} \left(x_{j} - \psi_{j}\right) 
	a_{j} \exp\left( \alpha_{j} - \beta_{j} x_{j} \right) v_{j}
\tag{$\mathsf{INP}_1$}
\label{eqn:inp1} \\
\text{s.t.} \quad
& v_{j} \ \ = \ \ 
		\theta_{j} + \sum_{i \in \mathcal{J}} 
		\left[1 - a_{i} \exp\left( \alpha_{i} - \beta_{i} x_{i} \right) \right] \rho_{ij} v_{i}
		& \forall \ j \in \mathcal{J} \nonumber \\
& a_{j} \ \ \in \ \ [0,1]
		& \forall \ j \in \mathcal{J} \nonumber 
\end{align}
\end{subequations}
Consider a variable change $r_j = a_{j} v_{j} \ , \ j \in \mathcal{J}$. We have
\begin{subequations}
\begin{align}
G(x) \ \ = \ \ 
\max_{v, \ a} \quad
	& \sum_{j \in \mathcal{J}} \left(x_{j} - \psi_{j}\right) 
	r_{j} \exp\left( \alpha_{j} - \beta_{j} x_{j} \right) 
\tag{$\mathsf{INP}_2$}
\label{eqn:inp2} \\
\text{s.t.} \quad
& v_{j} \ \ = \ \ 
		\theta_{j} + \sum_{i \in \mathcal{J}} 
		\left[v_{i} - r_{i} \exp\left( \alpha_{i} - \beta_{i} x_{i} \right) \right] \rho_{ij} 
		& \forall \ j \in \mathcal{J} \nonumber \\
& r_{j} \ \ \in \ \ [0,v_{j}]
		& \forall \ j \in \mathcal{J} \nonumber 
\end{align}
\end{subequations}
Now, (\ref{eqn:inp2}) is a linear program with $2|\mathcal{J}|$ variables, which means $2|\mathcal{J}|$ constraints in (\ref{eqn:inp2}) need to be tight at optimality. 
Thus, for any $j \in \mathcal{J}$, either $r_{j} = v_{j}$, or $r_{j} = 0$. 
Enough to conclude. 

\end{proof}

 \newpage
 
%--------------------------------------------------
%--------------------------------------------------

\subsection{Proof of Lemma \ref{lem:sp3 optimal}}

{\color{purple}
\begin{manuallemma}{\ref{lem:sp3 optimal}}
(\ref{eqn:sp3}) has an optimal solution.
\end{manuallemma}
}

\begin{proof}
The dual problem of (\ref{eqn:sp3}) is be given by: 
\begin{subequations}
\begin{align}
\min_{\eta, \, \pi, \, \overline{\nu}, \, \underline{\nu}} \quad & \sum_{j \in \mathcal{J}} \theta_{j} \eta_{j} 
\tag{\textsf{$\mathsf{SD_1}$}} 
\label{eqn:sd1} \\
\text{s.t.} \quad
& \pi_{j1} - \eta_{j} + \sum_{i \in \mathcal{J}} \rho_{ji} \eta_{i} \ \ = \ \ 0
		& \forall \ j \in \mathcal{J} \label{eqn:sd1-v} \\
& \pi_{j2} - \alpha_{j} \pi_{j3} - \sum_{i \in \mathcal{J}} \rho_{ji} \eta_{i} 
		+ \overline{\nu}_{j} \overline{x}_{j} - \underline{\nu}_{j} \underline{x}_{j} 
		\ \ = \ \ \psi_{j}
		& \forall \ j \in \mathcal{J} \\
& \beta_{j} \pi_{j3} + \underline{\nu}_{j} - \overline{\nu}_{j} \ \ = \ \ - 1
		& \forall \ j \in \mathcal{J} \\
& \pi_{j} \ \ \in \ \ \mathcal{K}_{\exp}^*
		& \forall \ j \in \mathcal{J} \\
& \overline{\nu}_{j} \ \ \ge \ \ 0
		& \forall \ j \in \mathcal{J} \\
& \underline{\nu}_{j} \ \ \ge \ \ 0
		& \forall \ j \in \mathcal{J}
\end{align}
\end{subequations}
where
\begin{align*}
\mathcal{K}_{\text{exp}}^* \defi& \ \text{closure} \ \Big\{ (a_{1}', a_{2}', a_{3}') \ : \ a_{1}' \geq -a_{3}' \ e^{a_{2}'/a_{3}' - 1}
		\ , \ a_{1}' > 0 \ , \ a_{3}' < 0 \Big\} \\
=& \Big\{ (a_{1}', a_{2}', a_{3}') \ : \ a_{1}' \geq -a_{3}' \ e^{a_{2}'/a_{3}' - 1} \ , \ a_{1}' > 0 \ , \ a_{3}' < 0 \Big\} \cup
		\Big\{ (a_{1}', a_{2}', 0) \ : \ a_{1}' \geq 0 \ , \ a_{2}' \geq 0 \Big\}
\end{align*}
is the dual cone to $\mathcal{K}_{\text{exp}}$. 
In (\ref{eqn:sd1}) $\eta$ is the dual variable corresponding to (\ref{eqn:sp3-balance}), $\pi$ is the dual variables corresponding to (\ref{eqn:sp3-cone}), $\overline{\nu}$ os the dual variables corresponding to (\ref{eqn:sp3-upper}), and $\underline{\nu}$ is the dual variables corresponding to (\ref{eqn:sp3-lower}). 

The refined conic duality theorem (see \cite{ben2001lectures}) states that (\ref{eqn:sp3}) has an optimal solution if its dual problem is bounded and is ``strictly feasible'' (i.e. it has a feasible solution in the interior of the cones). 
Since (\ref{eqn:sp3}) is feasible, we know that the dual problem is bounded below. 
Thus, we only need to show that the dual problem has a feasible solution such that the conic variables are all in the interior of $\mathcal{K}_{\text{exp}}^*$. 

Indeed. Let
\begin{align*}
& \pi_{j1} \ \ = \ \ 1 + \exp(- \psi_{j} - \alpha_{j} - 1) \ \ > \ \ 0
		& \forall \ j \in \mathcal{J} 
\end{align*}
and let $\eta$ be the unique solution to (\ref{eqn:sd1-v}) given the values of $\pi_{j1}$. 
Recall that $I - \rho$ is non-singular and has nonnegative entries. 
Thus, $\eta = (1 - \rho)^{-1} \pi_{\cdot 1} \geq 0$. 
In addition, let
\begin{align*}
& \pi_{j2} \ \ = \ \ \psi_{j} + \alpha_{j} + \sum_{i \in \mathcal{J}} \rho_{ji} \eta_{i}
		\ \ \geq \ \ \psi_{j} + \alpha_{j}
		& \forall \ j \in \mathcal{J} \\
& \pi_{j3} \ \ = \ \ -1 \ \ < \ \ 0
		& \forall \ j \in \mathcal{J} \\
& \overline{\nu}_{j} \ \ = \ \ 0
		& \forall \ j \in \mathcal{J} \\
& \underline{\nu}_{j} \ \ = \ \ 0
		& \forall \ j \in \mathcal{J}
\end{align*}
Clearly, the constructed solution $(\eta, \pi, \overline{\nu}, \underline{\nu})$ is feasible to (\ref{eqn:sd1}). 
In additon, we have
\begin{align*}
\pi_{j1} \ \ &= \ \ 1 + \exp(- \psi_{j} - \alpha_{j} - 1) \\
&\geq \ \ 1 + \exp(- \pi_{j2} - 1) \\
&= \ \ 1 - \pi_{j3} \cdot 
		\exp\left(\frac{\pi_{j2}}{\pi_{j3}} - 1\right) \\
&> \ \ - \pi_{j3} \cdot 
		\exp\left(\frac{\pi_{j2}}{\pi_{j3}} - 1\right) 
\end{align*}
Thus, $\pi_{j}$ is in the interior of $\mathcal{K}_{\text{exp}}^*$ for all $j \in \mathcal{J}$.

\end{proof}

 \newpage

%--------------------------------------------------
%--------------------------------------------------

\subsection{Proof of Theorem \ref{thm:rm-constant-price}}

Without loss of generality, 

{\color{purple}
\begin{manualtheorem}{\ref{thm:rm-constant-price}}
	The optimal objective value of (\ref{eqn:cbp1}) does not exceed that of (\ref{eqn:cbp2}). 
\end{manualtheorem}
}

\begin{proof}

Withou loss of generality, we can rescale (\ref{eqn:cbp1}) into:
\begin{subequations}
\begin{align}
\max_{\mathcal{A}(t), \ v(t), \ x(t)} \quad
	& \overline{\lambda} \int_{t=0}^{1} 
			\left( \sum_{j \in \mathcal{A}(t)} \left(x_{j}(t) - \psi_{j}\right) 
			\exp\left( \alpha_{j} - \beta_{j} x_{j}(t) \right) v_{j}(t)  \right) \diff t
\tag{$\mathsf{CBP}_1'$}
\label{eqn:cbp1p} \\
\text{s.t.} \quad
& \overline{\lambda} \int_{t=0}^{1} \left(\sum_{j \in \mathcal{A}(t)} \phi_{rj} 
		\exp\left( \alpha_{j} - \beta_{j} x_{j}(t) \right) v_{j}(t) \right) \diff t
		\ \ \leq \ \ \varphi_{r}
		& \forall \ r \in \mathcal{R} \nonumber \\
& v_{j}(t)  \ \ = \ \ 
		\theta_{j} + \sum_{i \in \mathcal{J}} 
		\left[1 - \mathbf{1}(i \in \mathcal{A}(t)) \exp\left( \alpha_{i} - \beta_{i} x_{i}(t) \right) \right] 
		\rho_{ij} v_{i}(t) 
		& \forall \ j \in \mathcal{J} \ , \ t \in [0, 1] \nonumber \\
& x_{j}(t) \ \ \leq \ \ \overline{x}_{j}
		& \forall \ j \in \mathcal{J} \ , \ t \in [0, 1] \nonumber \\
& x_{j}(t) \ \ \geq \ \ \underline{x}_{j}
		& \forall \ j \in \mathcal{J} \ , \ t \in [0, 1] \nonumber \\
& \mathcal{A}(t) \ \ \subseteq \ \ \mathcal{J}
		& \forall \ t \in [0, 1] \nonumber 
\end{align}
\end{subequations}
Consider any solution $(\mathcal{A}'(t), v'(t), x'(t))$ that is feasible to (\ref{eqn:cbp1p}). Let 
\begin{align*}
v_{j}^*(t) \ \ &= \ \ v_{j}^{\mathcal{A}} \\
x_{j}^*(t) \ \ &= \ \ x_{j}^{\mathcal{A}} 
\end{align*}
for every $\mathcal{A} \subseteq \mathcal{J}$, and $j, t$ such that $j \in \mathcal{J}$ and $\mathcal{A}'(t) = \mathcal{A}$, where
\begin{align*}
v_{j}^{\mathcal{A}} \ \ &= \ \ \frac{1}{\left|\{s : \mathcal{A}'(s) = \mathcal{A}\}\right|}
		\int_{s \in \{s : \mathcal{A}'(s) = \mathcal{A}\}} v_{j}'(s) \ \diff s \\
x_{j}^{\mathcal{A}} \ \ &= \ \ \frac{\alpha_{j}}{\beta_{j}} - \frac{1}{\beta_{j}} \ln \left( 
		\frac{\int_{s \in \{s : \mathcal{A}'(s) = \mathcal{A}\}}
		\exp\left( \alpha_{j} - \beta_{j} x_{j}'(s) \right) v_{j}'(s) \ \diff s}
		{\int_{s \in \{s : \mathcal{A}'(s) = \mathcal{A}\}} v_{j}'(s) \ \diff s} \right) 
\end{align*}
Clearly, in this new solution $(\mathcal{A}'(t), v^*(t), x^*(t))$, the values of $v^*(t)$ and $x^*(t)$ are constants under each distinct value of $\mathcal{A}'(t)$. 
We will show that $(\mathcal{A}'(t), v^*(t), x^*(t))$ is also feasible to (\ref{eqn:cbp1p}), and the objective value of (\ref{eqn:cbp1p}) at $(\mathcal{A}'(t), v^*(t), x^*(t))$ is equal to or better than that at $(\mathcal{A}'(t), v'(t), x'(t))$. 
Indeed. We have
\begin{align}
\begin{split}
& \overline{\lambda} \int_{t=0}^{1} \left(\sum_{j \in \mathcal{A}'(t)} \phi_{rj} 
		\exp\left( \alpha_{j} - \beta_{j} x_{j}'(t) \right) v_{j}'(t) \right) \diff t \\
& = \ \ \overline{\lambda} \sum_{\mathcal{A} \subseteq \mathcal{J}} \sum_{j \in \mathcal{A}} \phi_{rj} 
		\left( \int_{t \in \{t : \mathcal{A}'(t) = \mathcal{A}\}}
		\exp\left( \alpha_{j} - \beta_{j} x_{j}'(t) \right) v_{j}'(t) \diff t \right) \\
& = \ \ \overline{\lambda} \sum_{\mathcal{A} \subseteq \mathcal{J}} \sum_{j \in \mathcal{A}} \phi_{rj} 
		\exp\left( \alpha_{j} - \beta_{j} x_{j}^{\mathcal{A}} \right) 
		\left( \int_{t \in \{t : \mathcal{A}'(t) = \mathcal{A}\}} v_{j}'(t) \diff t \right) \\
& = \ \ \overline{\lambda} \sum_{\mathcal{A} \subseteq \mathcal{J}} \sum_{j \in \mathcal{A}} \phi_{rj} 
		\exp\left( \alpha_{j} - \beta_{j} x_{j}^{\mathcal{A}} \right) 
		\left( \int_{t \in \{t : \mathcal{A}'(t) = \mathcal{A}\}} v_{j}^*(t) \diff t \right) \\
& = \ \ \overline{\lambda} \sum_{\mathcal{A} \subseteq \mathcal{J}} \sum_{j \in \mathcal{A}} \phi_{rj} 
		\left( \int_{t \in \{t : \mathcal{A}'(t) = \mathcal{A}\}}
		\exp\left( \alpha_{j} - \beta_{j} x_{j}^*(t) \right) v_{j}^*(t) \diff t \right) \\
& = \ \ \overline{\lambda} \int_{t=0}^{1} \left(\sum_{j \in \mathcal{A}'(t)} \phi_{rj} 
		\exp\left( \alpha_{j} - \beta_{j} x_{j}^*(t) \right) v_{j}^*(t) \right) \diff t 
\label{eqn:integral-cbp1}
\end{split}
\end{align}
Thus, the resource constraints hold at $(\mathcal{A}'(t), v^*(t), x^*(t))$. Similarly, we can show that the balance equations hold at $(\mathcal{A}'(t), v^*(t), x^*(t))$ by taking integrals on both sides of the equations over $t$. In addition, we have
\begin{align*}
\exp\left( \alpha_{j} - \beta_{j} x_{j}^*(t) \right) v_{j}^*(t)
\ \ &= \ \ \frac{1}{\left|\{s : \mathcal{A}'(s) = \mathcal{A}\}\right|}
		\exp\left( \alpha_{j} - \beta_{j} x_{j}^*(t) \right)
		\int_{s \in \{s : \mathcal{A}'(s) = \mathcal{A}\}} v_{j}'(s) \ \diff s \\
\ \ &= \ \ \frac{1}{\left|\{s : \mathcal{A}'(s) = \mathcal{A}\}\right|}
		\int_{s \in \{s : \mathcal{A}'(s) = \mathcal{A}\}}
		\exp\left( \alpha_{j} - \beta_{j} x_{j}'(s) \right) v_{j}'(s) \ \diff s \\
\ \ &\leq \ \ \frac{1}{\left|\{s : \mathcal{A}'(s) = \mathcal{A}\}\right|}
		\exp\left( \alpha_{j} - \beta_{j} \overline{x}_{j} \right)
		\int_{s \in \{s : \mathcal{A}'(s) = \mathcal{A}\}} v_{j}'(s) \ \diff s \\
\ \ &= \ \ \exp\left( \alpha_{j} - \beta_{j} \overline{x}_{j} \right) v_{j}^*(t) \\
\implies \ \ x_{j}^*(t) \ \ &\leq \ \ \overline{x}_{j}
\end{align*}
and similarly, 
\begin{align*}
x_{j}^*(t) \ \ &\geq \ \ \underline{x}_{j}
\end{align*}
for every $\mathcal{A} \subseteq \mathcal{J}$, and $j, t$ such that $j \in \mathcal{J}$ and $\mathcal{A}'(t) \subseteq \mathcal{A}$. 
Thus, $(\mathcal{A}'(t), v^*(t), x^*(t))$ is feasible to (\ref{eqn:cbp1p}). 
It left to show that $(\mathcal{A}'(t), v^*(t), x^*(t))$ has a similar or better objective value than $(\mathcal{A}'(t), v'(t), x'(t))$. 
Similar to (\ref{eqn:integral-cbp1}), we have
\begin{align*}
& \int_{t=0}^{1} \left(\sum_{j \in \mathcal{A}'(t)} \psi_{rj} 
		\exp\left( \alpha_{j} - \beta_{j} x_{j}'(t) \right) v_{j}'(t) \right) \diff t 
\ \ = \ \ \int_{t=0}^{1} \left(\sum_{j \in \mathcal{A}'(t)} \psi_{rj} 
		\exp\left( \alpha_{j} - \beta_{j} x_{j}^*(t) \right) v_{j}^*(t) \right) \diff t 
\end{align*}
Meanwhile, for each $j \in \mathcal{J}$, let $d_{j}'(t) \defi \exp\left( \alpha_{j} - \beta_{j} x_{j}'(t) \right) v_{j}'(t)$ for any $t \in [0,1]$, and let $d_{j}^*(t) = \exp\left( \alpha_{j} - \beta_{j} x_{j}^*(t) \right) v_{j}^*(t)$ for any $\mathcal{A} \subseteq \mathcal{J}$. 
Thus, we have $ \alpha_{j} - \beta_{j} x_{j}'(t) = d_{j}'(t) \ln\left(d_{j}'(t) / v_{j}'(t)\right)$ if $v_{j}'(t) > 0$, and $\alpha_{j} - \beta_{j} x_{j}^*(t) =  d_{j}^*(t) \ln\left(d_{j}^*(t) / v_{j}^*(t)\right)$ if $v_{j}^*(t) > 0$. This implies: 
\begin{align*}
& \int_{s \in \{s : \mathcal{A}'(s) = \mathcal{A}\}} 
		\left( \alpha_{j} - \beta_{j} x_{j}'(s) \right)
		\exp\left( \alpha_{j} - \beta_{j} x_{j}'(s) \right) v_{j}'(s) \ \diff s \\
& \ \ = \ \ \int_{s \in \{s : \mathcal{A}'(s) = \mathcal{A}\}} 
		d_{j}'(s) \ln\left(\frac{d_{j}'(s)}{v_{j}'(s)}\right) \ \diff s \\
& \ \ \geq \ \ \left( \int_{s \in \{s : \mathcal{A}'(s) = \mathcal{A}\}} d_{j}'(s) \diff s \right)
		\ln\left(\frac{\int_{s \in \{s : \mathcal{A}'(s) = \mathcal{A}\}} d_{j}'(s) \diff s}
		{\int_{s \in \{s : \mathcal{A}'(s) = \mathcal{A}\}} v_{j}'(s) \diff s}\right) \\
& \ \ = \ \ d_{j}^*(t) 
		\ln\left(\frac{d_{j}^*(t)}{v_{j}^*(t)}\right)
		\ \ = \ \ \left( \alpha_{j} - \beta_{j} x_{j}^*(t) \right)
		\exp\left( \alpha_{j} - \beta_{j} x_{j}^*(t) \right) v_{j}^*(t) \\
& \int_{t \in \{t : \mathcal{A}'(s) = \mathcal{A}\}} 
		\left( \alpha_{j} - \beta_{j} x_{j}^*(t) \right)
		\exp\left( \alpha_{j} - \beta_{j} x_{j}^*(t) \right) v_{j}^*(t) \ \diff t 
\end{align*}
for every $\mathcal{A} \subseteq \mathcal{J}$, and $j, t$ such that $j \in \mathcal{J}$ and $\mathcal{A}'(t) \subseteq \mathcal{A}$. 
Thus, the objective value of (\ref{eqn:cbp1p}) at $(\mathcal{A}'(t), v^*(t), x^*(t))$ is equal to or better than that at $(\mathcal{A}'(t), v'(t), x'(t))$. 
In order words, there is no benefit of changing prices under each assortment. 

Now consider this feasible solution $(\mathcal{A}'(t), v^*(t), x^*(t))$ to (\ref{eqn:cbp1p}), where the values of $v^*(t)$ and $x^*(t)$ are constants under each distinct value of $\mathcal{A}'(t)$. 
Let $w^{\mathcal{A}} \defi \left|\{t : \mathcal{A}'(t) = \mathcal{A}\}\right|$ for every $\mathcal{A} \subseteq \mathcal{J}$. 
We construct a solution $(v,d,u)$ to (\ref{eqn:rp2}) as: 
\begin{align*}
v_{j} &= \sum_{\mathcal{A} \subseteq \mathcal{J}} w^{\mathcal{A}} v_{j}^{\mathcal{A}}
		& \forall \ j \in \mathcal{J} \\
d_{j} &= \sum_{\mathcal{A} \subseteq \mathcal{J}} 
		\mathbf{1} (j \in \mathcal{A}) w^{\mathcal{A}} 
		\exp\left( \alpha_{j} - \beta_{j} x_{j}^{\mathcal{A}} \right) v_{j}^{\mathcal{A}} 
		& \forall \ j \in \mathcal{J} \\
u_{j} &= \sum_{\mathcal{A} \subseteq \mathcal{J}} w^{\mathcal{A}} x_{j}^{\mathcal{A}} 
		\exp\left( \alpha_{j} - \beta_{j} x_{j}^{\mathcal{A}} \right) v_{j}^{\mathcal{A}}
		& \forall \ j \in \mathcal{J} 
\end{align*}
Clearly, the objective value of (\ref{eqn:rp2}) at $(v,d,u)$ is equal to that of (\ref{eqn:cbp1p}). 
The resources constraints, the balance equations, and the box constraints of (\ref{eqn:rp2}) all hold at $(v,d,u)$. 
In addition: 
\begin{align*}
d_{j} \ln\left( \frac{d_{j}}{v_{j}} \right) 
& \ \ \leq \ \ \sum_{\mathcal{A} \subseteq \mathcal{J} : \ j \in \mathcal{A}} w^{\mathcal{A}} 
		\exp\left( \alpha_{j} - \beta_{j} x_{j}^{\mathcal{A}} \right) v_{j}^{\mathcal{A}} 
		\ln\left( \frac{ \exp\left( \alpha_{j} - \beta_{j} x_{j}^{\mathcal{A}} \right) v_{j}^{\mathcal{A}}}{v_{j}^{\mathcal{A}}} \right) \\
& \ \ \leq \ \ \sum_{\mathcal{A} \subseteq \mathcal{J} : \ j \in \mathcal{A}} w^{\mathcal{A}} 
		\left( \alpha_{j} - \beta_{j} x_{j}^{\mathcal{A}} \right)
		\exp\left( \alpha_{j} - \beta_{j} x_{j}^{\mathcal{A}} \right) v_{j}^{\mathcal{A}} \\
& \ \ \leq \ \ \alpha_{j} d_{j} - \beta_{j} u_{j}
\end{align*}
for each $j \in \mathcal{J}$ such that $v_{j}, d_{j} > 0$. Thus, the conic constraints in (\ref{eqn:rp2}) also hold at $(v,d,u)$, which means $(v,d,u)$ is feasible to (\ref{eqn:rp2}), with an objective value as good as $(\mathcal{A}'(t), v^*(t), x^*(t))$ to (\ref{eqn:cbp1p}). 
Using the results in Section \ref{sec:network}, we can then construct a feasible solution to (\ref{eqn:cbp2}), with an objective value as good as $(\mathcal{A}'(t), v^*(t), x^*(t))$ to (\ref{eqn:cbp1p}). 

As a summary, the optimal objective value of (\ref{eqn:cbp1p}) does not exceed that of (\ref{eqn:cbp2}). 
\end{proof}

 \newpage

%--------------------------------------------------
%--------------------------------------------------

\subsection{Proof of Lemma \ref{lem:rp1-to-cbp2}}

{\color{purple}
\begin{manuallemma}{\ref{lem:rp1-to-cbp2}}
Let $(v, x, a)$ be a feasible solution to (\ref{eqn:rp1}). 
Using $(v, x, a)$ as the input, Algorithm \ref{algo:reduce} stops after at most $J+1$ iterations. 
Let $(\mathcal{A}_{k} : k = 1, \dots, K)$, $(v^{\mathcal{A}_{k}} : k = 1, \dots, K)$ and $(y^{\mathcal{A}_{k}} : k = 1, \dots, K)$ denote the output of Algorithm \ref{algo:reduce} using $(v, x, a)$ as the input, where $K$ is the number of iterations after which the algorithm stops. 
In addition, let 
\begin{align*}
w^{\mathcal{A}_{k}} &= (1 - y^{\mathcal{A}_1}) \cdots (1 - y^{\mathcal{A}_{k-1}}) y^{\mathcal{A}_{k}}
		\quad , \quad k = 1, \dots, K 
\end{align*}
let $w^{\mathcal{A}} = 0$ for any $\mathcal{A} \subseteq \mathcal{J}$ such that $\mathcal{A} \notin \{\mathcal{A}_1, \dots, \mathcal{A}_{k}\}$, and let $v^{\mathcal{A}}$ be the unique (see Remark \ref{rmk:balance-unique}) solution of the system
\begin{align*}
v_{j}^{\mathcal{A}}  
&\ \ = \ \ \theta_{j} + \sum_{i \in \mathcal{J}} 
		\left[1 - \mathbf{1}(i \in \mathcal{A}) \exp\left( \alpha_{i} - \beta_{i} x_{i} \right) \right] 
		\rho_{ij} v_{i}^{\mathcal{A}} 
		& \forall \ j \in \mathcal{J} 
\end{align*}
for every $\mathcal{A} \subseteq \mathcal{J}$ given the value of $x$. 
Then $(\mathbf{v}, x, w)$ is a feasible solution to (\ref{eqn:cbp2}) with an identical objective value as $(v, x, a)$ in (\ref{eqn:rp1}). \end{manuallemma}
}

{\color{blue}
Lemma \ref{lem:rp1-to-cbp2} can be proved by proving the following two lemmas: 
}

{\color{purple}
\begin{lemma}
\label{lem:rp1-to-cbp2-prep1}
Let $(v, x, a)$ be a feasible solution to (\ref{eqn:rp1}). Using $(v, x, a)$ as the input, the output from Algorithm \ref{algo:reduce} satisfies: 
\begin{align*}
y^{\mathcal{A}_{k}} \ \ &\in \ \ [0,1] \\
\hat{v}_{j}^{(k)} \ \ &= \ \ 
		\theta_{j} + \sum_{i \in \mathcal{J}} 
		\left[1 - \hat{a}_{i}^{(k)} \exp\left( \alpha_{i} - \beta_{i} x_{i} \right) \right] 
		\rho_{ij} \hat{v}_{i}^{(k)}
		& \forall \ j \in \mathcal{J} 
\end{align*}
at every step $k$. In addition, Algorithm \ref{algo:reduce} stops after at most $J+1$ iterations. 
\end{lemma}
}

\begin{proof}
First, consider the algorithm at step $1$. Since $(v, x, a)$ is a feasible solution to (\ref{eqn:rp1}), we have $\hat{v}_{j}^{(1)} = v_{j}$ and $\hat{a}_{j}^{(1)} = a_{j}$ for every $j \in \mathcal{J}$. That implies
\begin{align*}
\hat{v}_{j}^{(1)} \ \ &= \ \ 
		\theta_{j} + \sum_{i \in \mathcal{J}} 
		\left[1 - \hat{a}_{i}^{(1)} \exp\left( \alpha_{i} - \beta_{i} x_{i} \right) \right] \rho_{ij} \hat{v}_{i}^{(1)}
		& \forall \ j \in \mathcal{J} 
\end{align*}
In addition, we can show $y^{\mathcal{A}_{1}} \in [0,1]$ by contradiction. Clearly, $\hat{a}_{j}^{(1)} \hat{v}_{j}^{(1)} > 0$ and $v_{j}^{\mathcal{A}_{1}} > 0$ for every $j \in y^{\mathcal{A}_{1}}$. Thus, $y^{\mathcal{A}_{1}} > 0$. If $y^{\mathcal{A}_{1}} > 1$, then we must have $\mathcal{A}_{k} \neq \varnothing$, and $\hat{a}_{j}^{(1)} \hat{v}_{j}^{(1)} > v_{j}^{\mathcal{A}_{1}}$ for every $j \in y^{\mathcal{A}_{1}}$. However, since 
\begin{align*}
v_{j}^{\mathcal{A}_{1}}  
&\ \ = \ \ \theta_{j} + \sum_{i \in \mathcal{J}} 
		\left[1 - \mathbf{1}(i \in \mathcal{A}_{1}) \exp\left( \alpha_{i} - \beta_{i} x_{i} \right) \right] 
		\rho_{ij} v_{i}^{\mathcal{A}_{1}} 
		& \forall \ j \in \mathcal{J} 
\end{align*}
we have
\begin{align*}
\left( v_{j}^{\mathcal{A}_{1}} - \hat{v}_{j}^{(1)} \right)
& \ \ = \ \ \sum_{i \in \mathcal{J}} \rho_{ij} 
		\left( v_{i}^{\mathcal{A}_{1}} - \hat{v}_{i}^{(1)} \right)
		- \sum_{i \in \mathcal{A}_{1}} \exp\left( \alpha_{i} - \beta_{i} x_{i} \right) 
		\rho_{ij} \left( v_{i}^{\mathcal{A}_{1}} - \hat{a}_{i}^{(1)} \hat{v}_{i}^{(1)} \right)
		& \forall \ j \in \mathcal{J} 
\end{align*}
If $\mathcal{A}_{k} \neq \varnothing$, and $\hat{a}_{j}^{(1)} \hat{v}_{j}^{(1)} > v_{j}^{\mathcal{A}_{1}}$ for every $j \in y^{\mathcal{A}_{1}}$, then
\begin{align*}
\sum_{i \in \mathcal{A}_{1}} \exp\left( \alpha_{i} - \beta_{i} x_{i} \right) 
		\rho_{ij} \left( v_{i}^{\mathcal{A}_{1}} - \hat{a}_{i}^{(1)} \hat{v}_{i}^{(1)} \right)
\ \ \leq \ \ 0
\end{align*}
which implies $v_{j}^{\mathcal{A}_{1}} - \hat{v}_{j}^{(1)} \geq 0$ for every $j \in \mathcal{J}$ (recall that $(I - \rho)^{-1} = I + \sum_{n=1}^{\infty} (\rho)^{n} \geq 0$). That contradicts the fact that $\hat{a}_{j}^{(1)} = a_{j} \in [0,1]$ for every $j \in \mathcal{J}$. 

Thus, the conditions in Lemma \ref{lem:rp1-to-cbp2-prep1} hold for $\mathcal{A}_{1}$ and $\hat{v}^{(1)}$. 

Second, suppose that the conditions in Lemma \ref{lem:rp1-to-cbp2-prep1} hold for $\mathcal{A}_{k}$ and $\hat{v}^{(k)}$ at some $k \geq 1$, and the algorithm hasn't stop. 
We show that $\mathcal{A}_{k+1} \subseteq \mathcal{A}_{k}$. If $\mathcal{A}_{k} = \mathcal{J}$, then $\mathcal{A}_{k+1} \subseteq \mathcal{A}_{k}$ holds. 
Otherwise, consider any $j \in \mathcal{J} \setminus \mathcal{A}_{k}$. 
Then $\hat{a}_{i}^{(k+1)} = 0$, which means $\hat{a}_{i}^{(k+1)} \hat{v}_{i}^{(k+1)} = 0$. 
Thus, $j \notin \mathcal{A}_{k+1}$. 

Third, suppose that the condition in Lemma \ref{lem:rp1-to-cbp2-prep1} holds for $\mathcal{A}_{k}$ and $\hat{v}^{(k)}$ at some $k \geq 1$, and the algorithm hasn't stop. 
Since $\mathcal{A}_{k+1} \subseteq \mathcal{A}_{k}$, we have
{\scriptsize
\begin{align*}
\hat{v}_{j}^{(k+1)} 
\ \ &= \ \ \frac{\hat{v}_{j}^{(k)} - y^{\mathcal{A}_{k}} v_{j}^{\mathcal{A}_{k}}}
		{1 - y^{\mathcal{A}_{k}}} \\
\ \ &= \ \ \frac{\theta_{j} + \sum_{i \in \mathcal{J}} 
		\left[1 - \hat{a}_{i}^{(k)} \exp\left( \alpha_{i} - \beta_{i} x_{i} \right) \right] 
		\rho_{ij} \hat{v}_{i}^{(k)} 
		- y^{\mathcal{A}_{k}} v_{j}^{\mathcal{A}_{k}}}
		{1 - y^{\mathcal{A}_{k}}} \\
\ \ &= \ \ \frac{\theta_{j} + \sum_{i \in \mathcal{J}} \rho_{ij} \hat{v}_{i}^{(k)} 
		- \left(\sum_{i \in \mathcal{J}} \exp\left( \alpha_{i} - \beta_{i} x_{i} \right) 
		\rho_{ij} \left( \hat{a}_{i}^{(k)} \hat{v}_{i}^{(k)} \right) \right)
		- y^{\mathcal{A}_{k}} v_{j}^{\mathcal{A}_{k}}}
		{1 - y^{\mathcal{A}_{k}}} \\
\ \ &= \ \ \frac{\theta_{j} + \sum_{i \in \mathcal{J}} \rho_{ij} \hat{v}_{i}^{(k)} 
		- \left(\sum_{i \in \mathcal{A}_{k}} \exp\left( \alpha_{i} - \beta_{i} x_{i} \right) 
		\rho_{ij} \left( \hat{a}_{i}^{(k)} \hat{v}_{i}^{(k)} \right) \right)
		- y^{\mathcal{A}_{k}} v_{j}^{\mathcal{A}_{k}}}
		{1 - y^{\mathcal{A}_{k}}} \\
\ \ &= \ \ \frac{\theta_{j}}{1 - y^{\mathcal{A}_{k}}} 
		+ \sum_{i \in \mathcal{J}} \rho_{ij} 
		\left( \frac{\hat{v}_{i}^{(k)}}{1 - y^{\mathcal{A}_{k}}} \right) 
		- \left(\sum_{i \in \mathcal{A}_{k}} \exp\left( \alpha_{i} - \beta_{i} x_{i} \right) 
		\rho_{ij} \left( \frac{\hat{a}_{i}^{(k)} \hat{v}_{i}^{(k)}}{1 - y^{\mathcal{A}_{k}}} \right)\right)
		- \frac{y^{\mathcal{A}_{k}} v_{j}^{\mathcal{A}_{k}}}{1 - y^{\mathcal{A}_{k}}} \\
\ \ &= \ \ \frac{\theta_{j}}{1 - y^{\mathcal{A}_{k}}}
		+ \sum_{i \in \mathcal{J}} \rho_{ij} 
		\left(\frac{y^{\mathcal{A}_{k}} v^{\mathcal{A}_{k}}}{1 - y^{\mathcal{A}_{k}}} 
		+ \hat{v}_{i}^{(k+1)} \right)
		+ \left(\sum_{i \in \mathcal{A}_{k}} \exp\left( \alpha_{i} - \beta_{i} x_{i} \right) 
		\rho_{ij} \left(\frac{y^{\mathcal{A}_{k}} v^{\mathcal{A}_{k}}}{1 - y^{\mathcal{A}_{k}}} 
		+ \hat{a}_{i}^{(k+1)} \hat{v}_{i}^{(k+1)} \right)\right) 
		- \frac{y^{\mathcal{A}_{k}} v_{j}^{\mathcal{A}_{k}}}{1 - y^{\mathcal{A}_{k}}} \\
\ \ &= \ \ \frac{\theta_{j}}{1 - y^{\mathcal{A}_{k}}}
		- \left(\frac{y^{\mathcal{A}_{k}}}{1 - y^{\mathcal{A}_{k}}}\right) v^{\mathcal{A}_{k}}
		+ \left(\frac{y^{\mathcal{A}_{k}}}{1 - y^{\mathcal{A}_{k}}}\right) 
		\sum_{i \in \mathcal{J}} \left[1 - \mathbf{1}(i \in \mathcal{A}_{k}) 
		\exp\left( \alpha_{i} - \beta_{i} x_{i} \right) \right] 
		\rho_{ij} v_{i}^{\mathcal{A}_{k}} \\
\ \ &+ \ \ \sum_{i \in \mathcal{J}} \rho_{ij} \hat{v}_{i}^{(k+1)} 
		- \sum_{i \in \mathcal{A}_{k}} \exp\left( \alpha_{i} - \beta_{i} x_{i} \right) 
		\rho_{ij} \left( \hat{a}_{i}^{(k+1)} \hat{v}_{i}^{(k+1)} \right) \\
\ \ &= \ \ \frac{\theta_{j}}{1 - y^{\mathcal{A}_{k}}}
		- \frac{y^{\mathcal{A}_{k}}\theta_{j}}{1 - y^{\mathcal{A}_{k}}} 
		+ \sum_{i \in \mathcal{J}} 
		\left[1 - \hat{a}_{i}^{(k+1)} \exp\left( \alpha_{i} - \beta_{i} x_{i} \right) \right] 
		\rho_{ij} \hat{v}_{i}^{(k+1)} \\
\ \ &= \ \ \theta_{j} + \sum_{i \in \mathcal{J}} 
		\left[1 - \hat{a}_{i}^{(k+1)} \exp\left( \alpha_{i} - \beta_{i} x_{i} \right) \right] 
		\rho_{ij} \hat{v}_{i}^{(k+1)} 
\end{align*}
}%
for every $j \in \mathcal{J}$.

 With the same technique used in the first step of this proof, we can then show that $y^{\mathcal{A}_{k}} \in [0,1]$. Thus, the conditions in Lemma \ref{lem:rp1-to-cbp2-prep1} for $\mathcal{A}_{k+1}$ and $\hat{v}^{(k+1)}$. By induction, the conditions in Lemma \ref{lem:rp1-to-cbp2-prep1} holds for outputs of the algorithm at every step. 

Finally, recall that we have shown $\mathcal{A}_{k+1} \subseteq \mathcal{A}_{k}$ for $k = 1, 2, \cdots $. On the other hand, it is clear that $\mathcal{A}_{k+1} \neq \mathcal{A}_{k}$ as at least one of $j \in \mathcal{A}_{k}$ will be removed from $\mathcal{A}_{k+1}$. (That is, $\hat{a}_{j}^{(k)} \hat{v}_{j}^{(k)} = y^{\mathcal{A}_{k}} v_{j}^{\mathcal{A}_{k}}$ for some $j$ such that $y^{\mathcal{A}_{k}} = \min\{\hat{a}_{j}^{(k)} \hat{v}_{j}^{(k)} / v_{j}^{\mathcal{A}_{k}} : j \in \mathcal{A}_{k}\} = \hat{a}_{j}^{(k)} \hat{v}_{j}^{(k)} / v_{j}^{\mathcal{A}_{k}}$.) Thus, Algorithm \ref{algo:reduce} stops after at most $J+1$ iterations. 
\end{proof}

{\color{purple}
\begin{lemma}
\label{lem:rp1-to-cbp2-prep2}
Let $(v, x, a)$ be a feasible solution to (\ref{eqn:rp1}). 
Let $(\mathcal{A}_{k} : k = 1, \dots, K)$, $(v^{\mathcal{A}_{k}} : k = 1, \dots, K)$ and $(y^{\mathcal{A}_{k}} : k = 1, \dots, K)$ denote the output of Algorithm \ref{algo:reduce} using $(v, x, a)$ as the input, where $K$ is the number of iterations after which the algorithm stops. 
In addition, let 
\begin{align*}
w^{\mathcal{A}_{k}} &= (1 - y^{\mathcal{A}_1}) \cdots (1 - y^{\mathcal{A}_{k-1}}) y^{\mathcal{A}_{k}}
		\quad , \quad k = 1, \dots, K 
\end{align*}
let $w^{\mathcal{A}} = 0$ for any $\mathcal{A} \subseteq \mathcal{J}$ such that $\mathcal{A} \notin \{\mathcal{A}_1, \dots, \mathcal{A}_{k}\}$, and let $v^{\mathcal{A}}$ be the unique (see Remark \ref{rmk:balance-unique}) solution of the system
\begin{align*}
v_{j}^{\mathcal{A}}  
&\ \ = \ \ \theta_{j} + \sum_{i \in \mathcal{J}} 
		\left[1 - \mathbf{1}(i \in \mathcal{A}) \exp\left( \alpha_{i} - \beta_{i} x_{i} \right) \right] 
		\rho_{ij} v_{i}^{\mathcal{A}} 
		& \forall \ j \in \mathcal{J} 
\end{align*}
for every $\mathcal{A} \subseteq \mathcal{J}$ given the value of $x$. 
Then $(\mathbf{v}, x, w)$ is a feasible solution to (\ref{eqn:cbp2}) with an identical objective value as $(v, x, a)$ in (\ref{eqn:rp1}). 
\end{lemma}
}

\begin{proof}
First, the box constraints of prices $x$ in (\ref{eqn:cbp2}) and (\ref{eqn:rp1}) are identical. Thus,  the box constraints in (\ref{eqn:cbp2}) hold at $(\mathbf{v}, x, w)$. In addition, $\mathbf{v}$ and $x$ satisfies the balance equations in (\ref{eqn:cbp2}) by construction. 
Second, we show that $w \geq 0$ and $\sum_{\mathcal{A} \subseteq \mathcal{J}} w^{\mathcal{A}} = 1$. 
By Lemma \ref{lem:rp1-to-cbp2-prep1}, Algorithm \ref{algo:reduce} stops after at most $J+1$ iterations. Thus, $K \leq J+1$ and $y^{\mathcal{A}_{K}} = 1$. In addition, $y^{\mathcal{A}_{k}} \in [0,1]$ for every $k = 1, \dots, K$. Thus, $w \geq 0$ by construction. We have: 
{\small
\begin{align*}
\sum_{\mathcal{A} \subseteq \mathcal{J}} w^{\mathcal{A}}
& \ \ = \ \ \sum_{k=1}^K w^{\mathcal{A}_{k}} \\
& \ \ = \ \ \sum_{k=1}^K 
		(1 - y^{\mathcal{A}_1}) \cdots (1 - y^{\mathcal{A}_{k-1}}) y^{\mathcal{A}_{k}} \\
& \ \ = \ \ (1 - y^{\mathcal{A}_1}) \cdots (1 - y^{\mathcal{A}_{K-1}}) y^{\mathcal{A}_{K}}
		+ \sum_{k=1}^{K-1} 
		(1 - y^{\mathcal{A}_1}) \cdots (1 - y^{\mathcal{A}_{k-1}}) y^{\mathcal{A}_{k}} \\
& \ \ = \ \ (1 - y^{\mathcal{A}_1}) \cdots (1 - y^{\mathcal{A}_{K-1}}) + \sum_{k=1}^{K-1} 
		(1 - y^{\mathcal{A}_1}) \cdots (1 - y^{\mathcal{A}_{k-1}}) y^{\mathcal{A}_{k}} \\
& \ \ = \ \ (1 - y^{\mathcal{A}_1}) \cdots (1 - y^{\mathcal{A}_{K-2}}) + \sum_{k=1}^{K-2} 
		(1 - y^{\mathcal{A}_1}) \cdots (1 - y^{\mathcal{A}_{k-1}}) y^{\mathcal{A}_{k}} \\
& \ \ = \ \ \cdots \\
& \ \ = \ \ 1 
\end{align*}
}%
Third, we show that the resource constriaints in (\ref{eqn:cbp2}) still hold. We have: 
\begin{align*}
a_{j} v_{j} 
& \ \ = \ \ \hat{a}_{j}^{(1)} \hat{v}_{j}^{(1)} \\
& \ \ = \ \ (1-y^{\mathcal{A}_1}) \hat{a}_{j}^{(2)} \hat{v}_{j}^{(2)} - y^{\mathcal{A}_1} v_{j}^{\mathcal{A}_1} \\
& \ \ = \ \ \cdots \\
& \ \ = \ \ \sum_{k=1}^K  
		\mathbf{1}(j \in \mathcal{A}_{k}) w^{\mathcal{A}_{k}} v_{j}^{\mathcal{A}_{k}} 
	\ \ = \ \ \sum_{\mathcal{A} \subseteq \mathcal{J}} 
		\mathbf{1}(j \in \mathcal{A}) w^{\mathcal{A}} v_{j}^{\mathcal{A}} 
\end{align*}
which implies
{\small
\begin{align*}
& \ \ \sum_{\mathcal{A} \subseteq \mathcal{J}} \overline{\lambda} w^{\mathcal{A}} 
		\left(\sum_{j \in \mathcal{A}} \phi_{rj} \exp\left( \alpha_{j} - \beta_{j} x_{j} \right) 
		v_{j}^{\mathcal{A}} \right) \\
& \ \ = \ \ \sum_{\mathcal{A} \subseteq \mathcal{J}} \overline{\lambda} w^{\mathcal{A}} 
		\left(\sum_{j \in \mathcal{J}} \phi_{rj} \exp\left( \alpha_{j} - \beta_{j} x_{j} \right) 
		\mathbf{1}(j \in \mathcal{A}) v_{j}^{\mathcal{A}} \right) \\
& \ \ = \ \ \overline{\lambda} 
		\left( \sum_{j \in \mathcal{J}} \phi_{rj} \exp\left( \alpha_{j} - \beta_{j} x_{j} \right) 
		\left( \sum_{\mathcal{A} \subseteq \mathcal{J}} \mathbf{1}(j \in \mathcal{A}) 
		w^{\mathcal{A}} v_{j}^{\mathcal{A}} \right) \right) \\
& \ \ = \ \ \overline{\lambda} 
		\left( \sum_{j \in \mathcal{J}} \phi_{rj} \exp\left( \alpha_{j} - \beta_{j} x_{j} \right) 
		\left( a_{j} v_{j} \right) \right) \\
& \ \ = \ \ \overline{\lambda} \left(\sum_{j \in \mathcal{J}} \phi_{rj} 
		a_{j} \exp\left( \alpha_{j} - \beta_{j} x_{j} \right) v_{j} \right) \\
& \ \ \leq \ \ \varphi_{r}
\end{align*}
}%
Thus, $(\mathbf{v}, x, w)$ is a feasible solution to (\ref{eqn:cbp2}). 
Finally, with the same technique used in the third step, we can show that 
(\ref{eqn:cbp2}) has an identical objective value at $(\mathbf{v}, x, w)$ as (\ref{eqn:rp1}) at $(v, x, a)$. 
\end{proof}

 \newpage

%--------------------------------------------------
%--------------------------------------------------

\subsection{Proof of Lemma \ref{lem:rp2 optimal}}

{\color{purple}
\begin{manuallemma}{\ref{lem:rp2 optimal}}
If (\ref{eqn:rp1}) is feasible, then (\ref{eqn:rp2}) has an optimal solution.
\end{manuallemma}
}

\begin{proof}
Let 

Let $(v,x,a)$ be a feasible solution to (\ref{eqn:rp1}). Let 
\begin{align*}
u_{j} \ \ &= \ \ x_{j} d_{j}
		& \forall \ j \in \mathcal{J} \\
d_{j} \ \ &= \ \ a_{j} \exp\left( \alpha_{j} - \beta_{j} x_{j} \right) v_{j}
		& \forall \ j \in \mathcal{J} 
\end{align*}
Then $(v,d,u)$ is clearly a feasible solution to (\ref{eqn:rp2}). 
Thus, (\ref{eqn:rp2}) is feasible if (\ref{eqn:rp1}) is feasible

Now, since (\ref{eqn:rp2}) is feasible, the dual of (\ref{eqn:rp2}) is bounded.
Note that (\ref{eqn:rp2}) can be obtained by adding resource constraints to (\ref{eqn:sp3}), and then multiply its objective function by a scalar. 
Thus, the dual of (\ref{eqn:rp2}) has additional signed variables beyond (\ref{eqn:sd1}).
Then a feasible solution of the dual of (\ref{eqn:sp3}) in the interior of the dual cones, combined with these additional signed variables set to $0$, gives a feasible solution of the dual of (\ref{eqn:rp2}) in the interior of the dual cones.
Then it follows from the conic duality theorem (Theorem 1.4.2 in \cite{ben2001lectures}) that (\ref{eqn:rp2}) has an optimal solution.
\end{proof}

 \end{appendix}

%%%%%%%%%%%%%%%%%%%%%%%%%%%%%%
%%%%%%%%%%%%%%%%%%%%%%%%%%%%%%

\end{document}